\documentclass[11pt]{article}
\usepackage{latexsym,amsfonts,amsmath,theorem,amssymb,stmaryrd,array,tabularx}
\usepackage{pstricks,graphicx}

\newtheorem{theorem}{Theorem}[section]
\newtheorem{lemma}[theorem]{Lemma}
\newtheorem{corollary}[theorem]{Corollary}

\newenvironment{proof}{\begin{trivlist}
   \item[\hskip\labelsep{\bf Proof.}]}{$\hfill\Box$\end{trivlist}}

{\theoremstyle{plain} \theorembodyfont{\rmfamily}

\newtheorem{remark}[theorem]{Remark}}
\newtheorem{proposition}[theorem]{Proposition}

\numberwithin{equation}{section}

\newcommand{\mask}[1]{{}}
\DeclareMathOperator{\diam}{diam}

\newcommand{\vq}{\vec{q}}
\newcommand{\vx}{\vec{x}}
\newcommand{\vv}{\vec{v}}

\newcommand{\vz}{\vec{z}}
\newcommand{\vy}{\vec{y}}

\newcommand{\cV}{\mathcal{V}}
\newcommand{\cW}{\mathcal{W}}
\newcommand{\rdiv}{\mathrm{div}\,}

\newcommand{\cT}{\mathcal{T}}
\newcommand{\cE}{\mathcal{E}}
\newcommand{\be}{\begin{equation}}
\newcommand{\ee}{\end{equation}}
\newcommand{\amax}{a_{\mathrm{max}}}
\newcommand{\amin}{a_{\mathrm{min}}}
\newcommand{\Hdiv}{H(\rdiv,D)}
\newcommand{\oD}{\overline{D}}

\newcommand{\Div}{\textnormal{div}}
\newcommand{\dx}{\, {\rm d} \vx}

\graphicspath{{./figures/}}

\title{Mixed Finite Element Analysis of Lognormal Diffusion 
and Multilevel Monte Carlo Methods}
\author{
I. G. Graham\footnotemark[1], \
R. Scheichl\footnotemark[1], \ and \ 
E. Ullmann\footnotemark[1]}
\date{\today}
\begin{document}
\maketitle
\renewcommand{\thefootnote}{\fnsymbol{footnote}}
\footnotetext[1]{Department of Mathematical Sciences, University of Bath, Claverton Down, Bath BA2 7AY, UK. Email: {\tt
I.G.Graham@bath.ac.uk, R.Scheichl@bath.ac.uk, E.Ullmann@bath.ac.uk}.}


 
\maketitle

\begin{abstract}
This work is motivated by the need to develop efficient tools for uncertainty quantification in subsurface flows
associated with radioactive waste disposal studies.
We consider single phase flow problems in random porous media described by correlated lognormal distributions.
We are interested in the error introduced by a finite element discretisation of these problems.
In contrast to several recent works on the analysis of standard nodal finite element discretisations, we consider
here mass-conservative lowest order Raviart-Thomas mixed finite elements. 
This is very important since local mass conservation is highly desirable in realistic
groundwater flow problems. 
Due to the limited spatial regularity and the lack of
uniform ellipticity and boundedness of the operator the analysis is non-trivial in the presence of lognormal random fields.
We establish finite element error bounds for Darcy velocity and pressure,
as well as for a more accurate recovered pressure approximation.
We then apply the error bounds to prove convergence of the multilevel Monte Carlo algorithm for estimating
statistics of these quantities.
Moreover, we prove convergence for a class of bounded, linear functionals of the Darcy velocity.
An important special case is the approximation of the effective permeability in a 2D flow cell.
We perform numerical experiments to confirm the convergence results.

\end{abstract}


\section{Introduction}

Interest in the analysis, discretisation and postprocessing of partial differential equations (PDEs) with random
coefficients has risen sharply over the past decade. 
These equations are used in computer simulations of physical processes with uncertain inputs in science, engineering and industry.
The goal is to obtain quantitative estimates of the effect of input data uncertainties in order to reliably evaluate
simulation results. 
Typical output quantities of interest are the expected value and higher moments of the solution or the probability of certain events which can be expressed as integrals over the input sample space.

In this paper we are concerned with the quantification of uncertainties in the simulation of subsurface flows.
This plays an important role, for instance, in the safety assessment of proposed long-term radioactive waste
repositories. 
To this end, we consider single phase fluid flow in a porous medium with random, correlated permeability. 
The flow is governed by Darcy's law with random permeability and conservation of mass.
For simplicity, the pressure/flow boundary data is assumed deterministic.
Mathematically, this problem can be formulated as an elliptic PDE with a random diffusion coefficient.
A standard discretisation for these equations with respect to the physical variables is by finite elements (FEs) where the solution to the PDE is approximated on meshes constructed in the spatial domain of interest.
The well-posedness and subsequent finite element error analysis of elliptic PDEs with sufficiently regular, uniformly positive and bounded random coefficients has been established over the past decade and is classical \cite{BNT:2007,BTZ:2004,BarthSchwabZollinger:2011,FST:2005,MatthiesKeese:2005}.

Unfortunately, these results cannot be applied to practical
subsurface flow situations where the permeability is modeled as a correlated, lognormal random field.
Notably, this model ensures positivity of the permeability field and also accounts for the large variations of the
permeability observed in real world data sets. 
However, lognormal diffusion coefficients are not bounded uniformly over all realisations and a non-standard approach
has to be taken to prove existence, uniqueness, and regularity results for the PDE solution.
Moreover, depending on the covariance function the trajectories of such permeability fields are often only
H\"{o}lder continuous with exponent $t \leq 1$ and thus the regularity of the solution is limited.
The analysis of the primal formulation of the lognormal flow problem can be found in
\cite{Charrier:2012,CST:2011,GalvisSarkis:2009,Gittelson:2010,MuglerStarkloff:2011,TSGU:2012}.

One important quantity of interest in radioactive waste disposal is the time it takes radionuclides in case of an accidental damage of the waste repository to travel from the (center of the) repository to the boundary of a well-defined safety zone.
To be useful for realistic groundwater flow problems, it is desirable to use locally mass-conservative
discretisation schemes. 
Discretisation schemes that do not have this property, such as standard Lagrange finite
elements, lead to unphysical approximations of the Darcy velocity and of particle trajectories, even for very simple
models of the particle transport. 
In the framework of finite element discretisations local mass-conservation can be
achieved by mixed finite element methods \cite{BBF:2013}. 

Thus the goal of this paper is to extend the finite element error analysis in \cite{CST:2011,TSGU:2012} to lowest order
Raviart-Thomas mixed elements \cite{RaviartThomas:1977}.
In particular, we consider linear functionals of the (Darcy) velocity which enables us to analyse, e.g., the FE error
for the effective permeability in a simple 2D flow cell.
In addition, we study the properties of a piecewise linear, discontinuous recovered pressure approximation
first introduced in \cite{ArnoldBrezzi:1985}.
This enjoys a faster convergence when the lognormal field has trajectories
with H\"{o}lder exponent $t > 1/2$.
For less smooth fields the standard piecewise constant pressure approximation converges with the same rate. 
This latter result is not surprising, but we could not find it in the classical literature. 
It relies on a duality argument similar to the ones used in \cite{DouglasRoberts:1985,FalkOsborn:1980}.

Mixed formulations of elliptic PDEs with lognormal coefficients are covered by the analysis (for the Brinkman problem)
in \cite{Gittelson_etal:2011}. However, in that work stabilised FEs are
used and the analysis is carried out under strong regularity assumptions on the
solution. 
Here, we make no such a priori assumption and instead deduce the regularity of the
solution from the regularity of the diffusion coefficient.
Moreover, we use lowest order Raviart-Thomas mixed finite elements together with
piecewise constant pressure elements which do not require stabilisation.
Mixed formulations of elliptic PDEs with random coefficients are also studied in \cite{Bespalov_etal:2012}, where an
error analysis in the framework of stochastic Galerkin discretisations is carried out. 
However, there the random coefficient is assumed uniformly bounded which is not the case for lognormal
coefficients.

The stochastic discretisation of PDEs with random coefficients often relies on Karhunen-Lo\`{e}ve expansions of the
random inputs \cite{GhanemSpanos:1991}. 
This is not the case in Monte Carlo type methods.
Karhunen-Lo\`{e}ve expansions can be used to transform a random PDE into a parametric one. However, the typically short
correlation length in subsurface flow problems results in very high-dimensional stochastic parameter spaces. 
Integration over this parameter space, the core task in uncertainty quantification, by spectral methods
\cite{BNT:2007,GhanemSpanos:1991,KnioLeMaitre:2010,Xiu:2010}, 
such as stochastic Galerkin or stochastic collocation methods, is at 
least up-to-now restricted to a few tens or hundreds of parameters.
To date, algorithms for really high-dimensional integration instead use Monte
Carlo (MC) based approaches since, crucially, the convergence rates of
 these methods are dimension independent.

In subsurface flow situations as described above we have rough, highly
variable, large contrast permeabilities. 
The standard MC estimator is in general very expensive in this context
since it is necessary to generate large numbers of samples and to
solve on very fine spatial meshes to obtain acceptable accuracies.
The Multilevel Monte Carlo (MLMC) method overcomes this difficulty by computing approximations to output statistics on a
hierarchy of meshes corresponding to the spatial discretisations and not only on a single mesh. 
MLMC was introduced by Heinrich \cite{Heinrich:2001} for the approximation of parameter-dependent integrals in high-dimensions. 
More recently it has been applied by Giles \cite{Giles:2008} in the context of stochastic differential equations and has
since been used for the approximation of output statistics of PDEs with random coefficients, see
\cite{BarthSchwabZollinger:2011,CGST:2011,Gittelson_etal:2011,Graubner:2008,MishraSchwabSukys:2012}. 

Another option to analyse the MLMC convergence is by interpreting this method as a sparse tensor product approximation \cite{Harbrecht_etal:2011}. 
However, we follow the works \cite{CGST:2011,TSGU:2012} where the convergence and complexity of MLMC is analysed in terms of the root mean square error of the MLMC estimator. 
This approach is based on assumptions on the decay of the expected value of the FE error and of the variance of the
difference of FE approximations on two consecutive grids, see \cite[Theorem 1]{CGST:2011}. 
These assumptions have been proved for the model elliptic PDE with
lognormal diffusion coefficients and standard Lagrange-type finite
elements in \cite{CST:2011} (see also \cite{aretha_thesis,TSGU:2012}
for further extensions).
Here, we extend the error analysis and thus also the convergence analysis for MLMC to mass-conservative, mixed finite
element schemes which are more suitable for subsurface flow applications. 
The results also form a crucial ingredient in the analysis of Quasi Monte Carlo (QMC) and multilevel QMC methods
for lognormal diffusion problems in mixed form \cite{gknsss:2012,Graham_etal:2011}.

The rest of this paper is organized as follows. 
In Section~\ref{sec:problem} we present the lognormal diffusion model problem in mixed form and its discretisation by lowest order Raviart-Thomas elements for the Darcy velocity and piecewise constant elements for the hydrostatic pressure. 
The regularity of the solution to the mixed formulation is a basic ingredient for the FE error analysis and is studied in Section~\ref{sec:reg}. 
Mixed FE error estimates for the Darcy velocity and for the pressure are derived in Section~\ref{sec:error}. 
Here, we also study two post-processing scenarios for Raviart-Thomas mixed methods of practial interest. 
We derive error estimates for a class of linear, bounded velocity functionals and
for a piecewise linear pressure recovery process first proposed in \cite{ArnoldBrezzi:1985}. 
In Section~\ref{sec:MLMC} we apply the error estimates in the complexity analysis of the multilevel Monte Carlo method
applied to our model problem.
Finally, in Section~\ref{sec:numerics} we confirm the theoretical results by numerical experiments. 

\paragraph{Notation} 
In the rest of this paper we will use ``$\lesssim \dots$'' to denote ``$\leq C \dots$'' where the generic constant $C>0$
is independent of the characteristic finite element mesh size, the approximated function and the random diffusion
coefficient. 
Constants which depend on a specific realisation of the random coefficient will be stated explicitly. 
This is necessary since our analysis follows closely the idea in \cite{CST:2011, TSGU:2012} where the finite element
error analysis is first performed for a fixed realisation.
Error estimates are then extended to the entire sample space using H\"{o}lder's inequality.

The FE error analysis of PDEs with deterministic coefficients is well established to date and standard textbook methods
can be used.
In contrast, the analysis of PDEs with random coefficients is non-trivial and requires considerable care.
Since the input data of the PDE is random it can happen that the constants in the standard error estimates
depend on a specific realisation or data set. 
Mathematically speaking, the constants are random variables. 
If we are interested in uniform error estimates on the entire sample space it is thus necessary to derive explicit expressions for these random variables.

\section{Lognormal diffusion: mixed formulation and regularity }\label{sec:problem}

We study a coupled first-order system of PDEs with random coefficient on a bounded, Lipschitz polygonal/polyhedral
domain $D \subset \mathbb{R}^d$, $d=2,3$, stated in mixed form:
\begin{eqnarray}\label{log-diff}
a^{-1}(\omega,\vx) \vq(\omega,\vx) + \nabla u(\omega,\vx) & = & \vec{g}(\omega,\vx), \label{pde1.1} 
\\
\rdiv \vq(\omega,\vx) & = & f(\omega,\vx) \qquad\;\; \textnormal{in } D  \label{pde1.2}\ . 
\end{eqnarray}

For the sake of a transparent presentation we assume deterministic
Dirichlet boundary conditions $u(\omega,\vx) = u_\Gamma(\vx) $ on the entire
boundary $\partial D$ of $D$. Note, however, that our analysis carries
through without much further work also to Neumann and mixed boundary
conditions, as well as to random boundary data. The test problem in Section~\ref{sec:numerics} will in
fact have a set of mixed boundary conditions.
Now, given a probability space $(\Omega,\mathcal{A},P)$ with sample space $\Omega$ and a probability measure $P$, we
require that $\eqref{pde1.1}-\eqref{pde1.2}$ as well as the boundary conditions are satisfied for all samples $\omega
\in \Omega$ $P$-almost surely ($P$-a.s.). 
Note that the $\sigma$-algebra $\mathcal{A}$ associated with $\Omega$ is generated by the collection of random
variables $\{a(\cdot,\vx) \colon \vx \in D\}$.

In the context of steady-state flow in a porous medium, $\vq$ is the Darcy velocity, $u$ is the hydrostatic pressure and $a$ is the permeability. 
The empirical relation between pressure and velocity \eqref{pde1.1} is known as Darcy's law and \eqref{pde1.2} is the law of conservation of mass. 
In this paper, the coefficient $a(\omega,\vx)$ is assumed to be a \textit{lognormal} random field, 
i.e. $\log a(\omega,\vx)$ is Gaussian with a certain mean $\mu(\vx)$ and covariance function
$C(\vx,\vy):=\mathbb{E}[(\log a(\cdot,\vx)-\mu(\vx))(\log a(\cdot,\vy)-\mu(\vy))]$.

For every sample $\omega\in\Omega$, we define
\[
a_{\min}(\omega) \,:=\,
\min_{\vx\in \overline{D}} \ a(\omega,\vx) \quad \text{and}
\quad a_{\max}(\omega) \,:=\, \max_{\vx\in \overline{D}} \ a(\omega,\vx).
\]
For these to be well-defined and to be able to use the regularity results proved in \cite{TSGU:2012}, we make certain
assumptions on the coefficient $a$, the source terms $\vec{g}$ and $f$ and the boundary data $u_\Gamma$. 
To this end, let $\mathcal{B}$ denote a Banach space with norm $\|\cdot\|_{\mathcal{B}}$.
Let $L^p(\Omega,\mathcal{B})$ denote the space 
of $\mathcal{B}$-valued random variables with finite $p^{\text{th}}$
moment (with respect to the probability measure $P$) of the $\mathcal{B}$-norm. 
For brevity we write $L^p(\Omega,\mathbb{R})=:L^p(\Omega)$. 
The space of H\"{o}lder-continuous functions with exponent $t$ is
denoted by $C^t(\overline{D})$; $H^s(D)$ is the Sobolev space of 
(fractional) order $s \in \mathbb{R}$ (see, e.g., \cite[Chapter 7]{Adams:2003}); 
$H(\rdiv,D)$ is the subspace of functions $\vv \in L^2(D)^d$ where 
$\rdiv \vv \in L^2(D)$ (see, e.g., \cite{BBF:2013}) with 
norm denoted by $\|\cdot\|_{H(\rdiv)}$. 
For fixed $\omega \in \Omega$, to simplify the presentation, we will write $v_\omega$ instead of 
$v(\omega,\cdot)$ for any function $v$ on $\Omega \times D$ and likewise $\vv_\omega$ instead of
$\vv(\omega,\cdot)$. 
With these definitions we can now state our basic assumptions:
\begin{itemize}
\item[{\bf A1.}] 
$a_{\min} > 0$ $P$-a.s. and $1/a_{\min} \in L^p(\Omega)$, for all $p \in [1,\infty)$,
\item[{\bf A2.}]\label{eq:mfeass2}
$a \in L^p(\Omega, C^t(\overline{D}))$, for some $0 < t
\leq 1 $ and all $p \in [1,\infty)$,
\item[{\bf A3.}] $u_\Gamma \in H^{1/2+t}(\partial D)$, 
$\vec{g} \in L^r(\Omega, H^{t}(D)^d)$,
$f \in L^r(\Omega, H^{t}(D))$
for some $0 < t \le 1$ and $r \in [1,\infty)$.\vspace{0.5ex}
\end{itemize}
Note A2 implies that $a_{\max}(\omega)$ is well-defined and $a_{\max}(\omega) \in
L^p(\Omega)$, for all $p \in [1,\infty)$.

Assumptions A1--A2 are satisfied for any lognormal diffusion coefficient $a$ where the underlying
Gaussian random field $\log(a)$ has a Lipschitz continuous, isotropic covariance function and a mean function that
belongs to $C^{t}(\overline{D})$ (cf. \cite[Proposition 2.4]{CST:2011}).
Moreover, ${1}/{a_{\min}}  \in L^p(\Omega)$ for all $p \in [1,\infty)$ is proved in \cite[Proposition
2.3]{Charrier:2012}.
The H\"{o}lder-regularity of the trajectories of $\log(a)$ (and thus of $a$) follows from the smoothness of its
covariance function.
For the exponential covariance function it can only be shown that trajectories $a_\omega \in C^t(\overline{D})$, for all
$t < 1/2 $ \cite[\S~2.3]{CST:2011}.
For Mat\'{e}rn covariances with $\nu \in (1/2,1)$ we have $a_\omega \in
C^t(\overline{D})$ for all $t < \nu $ \cite[\S~2.2]{gknsss:2012}.
For $\nu > 1$, the trajectories $a_\omega \in C^1(\overline{D})$ and thus A2 holds for $t = 1$
\cite[Remark 4]{gknsss:2012}. 

\begin{remark}
The assumption $f \in L^r(\Omega,H^t(D))$ in A3 can be generalised to $f \in L^r(\Omega,H^{\tilde{t}}(D))$ for
some $0\leq \tilde{t}\leq 1$ with $\tilde{t}\neq t$. We choose not to do this to simplify the
presentation.
\end{remark}

\subsection{Weak formulation}
To further analyse and to eventually discretise \eqref{pde1.1}-\eqref{pde1.2}, 
we put it in weak form (see, e.g. \cite{BBF:2013} for
details). 
Let $\omega \in \Omega$ be fixed and set $\cV \ = H(\mathrm{div},D)$ and $\cW =L^2(D)$.
Introducing, for all $\vec{\eta}, \vec{v} \in \cV$ and $w \in \cW$, the  bilinear
forms
\[
m_\omega(\vec{\eta},\vec{v}) :=
(a_\omega^{-1} \vec{\eta}, \vec{v})_{L^2(D)},
\qquad b(\vec{v},w) := -(\mathrm{div} \, \vec{v}, w)_{L^2(D)}, \
\]
and the linear functionals
\[
G_\omega(\vv) := \big(\vec{g}_\omega,\vv\big)_{L^2(D)} \, - \, \int_{\partial D} u_\Gamma \, \vv \cdot \vec{\nu} \, ds,
\qquad F_\omega(w) := -\big(f_\omega,w\big)_{L^2(D)}\ ,
\]
the weak form of
(\ref{pde1.1})-(\ref{pde1.2}) is to find
$(\vec{q}_\omega,u_\omega) \in \cV \times \cW $
such that $P$-a.s.
\begin{equation}
\left.
\begin{array}{rclll}
\hspace*{1.5cm} m_\omega(\vec{q}_\omega,\vec{v}) & + \ \ b(\vec{v},u_\omega)
& = & G_\omega(\vec{v}) &\mathrm{for\ all\ } \vec{v} \in \cV, \\[1ex]
b(\vec{q}_\omega,w)&  & = & F_\omega(w) &\mathrm{for\ all\ }
w \in \cW.
\end{array}
\right\}
\label{3.3}
\end{equation}

Existence and uniqueness results for problem \eqref{3.3} for fixed $\omega \in \Omega$ are classical (see,
e.g., \cite[Theorem 4.2.3]{BBF:2013}).
They rely on certain continuity, coercivity and inf-sup stability conditions being satisfied. 
In particular, we need all the bilinear forms and linear functionals to be bounded. 
The following bounds follow immediately from the Cauchy-Schwarz inequality
and a trace result:
\begin{equation}
\label{bounds}
\begin{split}
&\|m_\omega\|_{\mathcal{L}(\cV,\cV')} \le
\frac{1}{a_{\min}(\omega)}, \ \ \ \ \ \|b\|_{\mathcal{L}(\cV,\cW')} \le 1, \ \ \\
&\|F_\omega\|_{\cW'} \le \|f_\omega\|_{L^2(D)}\ \ \text{and} \ \ \ 
\|G_\omega\|_{\cV'} \le \|\vec{g}_\omega\|_{L^2(D)}+\|u_\Gamma\|_{H^{1/2}(\partial D)}\,.
\end{split}
\end{equation}
The inf-sup stability of $b$ is also classical, i.e.
\begin{equation}\label{infsup}
\sup_{\vv \in \cV} \frac{b(\vv , w )}{\Vert \vv \Vert_{H(\rdiv)}} \  \geq k_0 
\Vert w \Vert_{L^2(D)}, \quad \text{for all } \ w \in \cW,
\end{equation}
with a constant $k_0>0$ that is independent of $\omega$ and only depends on 
the shape of the domain $D$.
Finally, to establish the coercivity of $m_\omega$ let us 
introduce  
\begin{equation}
Z \ =  \ \{ \vv \in \cV : b(\vv , w ) = 0 \quad \text{for
  all} \quad w \in \cW\}\ .\notag
\end{equation}
This subspace of $\cV$ is called $\ker B$ in \cite{BBF:2013}.
The bilinear form $m_\omega(\cdot , \cdot)$ is coercive on $Z$, i.e.
\begin{equation}\label{coercive}
m_\omega(\vv , \vv ) \ \geq \ \frac{1}{a_{\max}(\omega)} \|\vv\|_{H(\mathrm{div})}^2, \qquad \text{for all } \ \vv \in Z.
\end{equation}

\begin{proposition}
\label{prop2.2}
Under the Assumptions A1--A3, the family of problems \eqref{3.3} has a 
unique solution $(\vq,u)$ with $\vq\in L^p(\Omega,\cV )$ and $u \in L^p(\Omega,\cW )$, for all \ $1 \le p < r$.
\end{proposition}
\begin{proof}
For $\omega \in \Omega$ fixed there is a unique solution $(\vq_\omega,u_\omega) \in \cV \times \cW$.
This follows from the continuity conditions \eqref{bounds}, together
with the inf-sup stability \eqref{infsup} of $b$ and the coercivity
\eqref{coercive} of $m_\omega$, since $0 < \frac{1}{a_{\max}(\omega)} \le \frac{1}{a_{\min}(\omega)} < \infty$ for
almost all $\omega \in \Omega$.
That $\vq \in L^p(\Omega,\cV)$ and $u \in L^p(\Omega,\cW)$ follows from standard stability estimates of
$\Vert \vq_\omega \Vert_{\cV}$ and $\Vert u_\omega \Vert_{\cW}$ (see, e.g., \cite[Theorem 4.2.3]{BBF:2013}) together
with \eqref{bounds}--\eqref{coercive}, Assumptions A1--A3, and the H\"{o}lder inequality.
\end{proof}

\subsection{Regularity of the solution}\label{sec:reg}

To prove convergence of finite element approximations to the solution of 
\eqref{3.3} we need to study the regularity of its solution.
Due to the following equivalence between primal and dual formulations of 
\eqref{pde1.1}-\eqref{pde1.2}, we can use the regularity results proved in 
\cite{CST:2011,TSGU:2012}.

\begin{lemma}
\label{equivalence}
Suppose Assumptions A1-A3 hold and 
$(\vq_\omega,u_\omega) \in \cV \times \cW$ is the unique 
solution of \eqref{3.3}, $P$-a.s. for $\omega \in \Omega$. Then 
$u_\omega \in H^1_{u_\Gamma}(D) := \{v \in H^1(D): v=u_\Gamma \text{ on } \partial D\}$ 
and 
 \begin{equation}
\label{primal-weak}
\Big(a_\omega \nabla u_\omega, \nabla \phi\Big)_{L^2(D)} = \Big(\tilde{f}_\omega,\phi\Big)_{L^2(D)} \qquad \text{for all } \ \phi \in 
H^1_0(D)\ ,
\end{equation}
where $\tilde{f}_\omega := f_\omega - \rdiv(a_\omega\vec{g}_\omega)$. 
Moreover, we have
\begin{equation}
\label{eq:H1bound}
\| u_\omega \|_{H^1(D)} \lesssim \frac{1}{a_{\min}(\omega)} \left(\|f_\omega\|_{H^{-1}(D)} + a_{\max}(\omega) \|
\vec{g}_\omega \|_{L^2(D)} \right)
+
|u_\Gamma|_{H^{1/2}(\partial D)}\,.
\end{equation}
\end{lemma}
\begin{proof}
Let us fix $\omega \in \Omega$.
Choose $\vv \in H(\rdiv,D)$ with $v_i \in C_0^\infty(D)$ and $v_j=0$, for $j \not=i$.
Then the first equation in \eqref{3.3} gives
\[
\int_D u_\omega \frac{\partial v_i}{\partial x_i} \dx = \int_D (a_\omega^{-1} (q_{\omega})_i - (g_{\omega})_i) v_i \dx
\,.
\]
This ensures the existence of weak first-order partial derivatives of $u_\omega$, such that 
\begin{equation}\label{uprime}
\nabla u_\omega = -(a_\omega^{-1} \vq_\omega - \vec{g}_\omega) 
\end{equation}
and $u_\omega \in H^1(D)$. 
Now, using the second equation in \eqref{3.3} and integrating by parts, for any $\phi \in H^1_0(D)$,
\[
\big(f_\omega,\phi\big)_{L^2(D)} 
= \big(\rdiv \vq_\omega,\phi\big)_{L^2(D)}
= -\big(\vq_\omega,\nabla \phi\big)_{L^2(D)} 
= \big(a_\omega \nabla u_\omega,\nabla \phi\big)_{L^2(D)} +
\big(\rdiv(a_\omega \vec{g}_\omega),\phi\big)_{L^2(D)}\, .
\]
Moreover, the first equation in \eqref{3.3} and Green's formula tell us that if 
$\vv \in (C^\infty(D))^d$ then
$$
\int_{\partial D} u_\omega \vv\cdot \vec{\nu} \,ds \, = \, \int_D \nabla
u_\omega \cdot \vv \, + \,  u_\omega \rdiv \vv \dx  \, = \, 
\int_D (\vec{g}_\omega-a_\omega^{-1}\vq_\omega) \cdot \vv 
\, + \, u_\omega \rdiv \vv \dx \, = \,
\int_{\partial D} u_\Gamma \vv\cdot \vec{\nu} \,ds \, .
$$
Hence, $u_\omega=u_\Gamma$ on all the smooth parts of $\partial D$. 
The extension to Lipschitz polygonal boundaries is classical \cite{McLean}.
In summary, 
$u_\omega \in H^1_{u_\Gamma}(D)$ and by \eqref{uprime} it satisfies \eqref{primal-weak}. 
Since
$$
\|\tilde{f}_\omega\|_{H^{-1}(D)} = 
\|f_\omega - \rdiv(a_\omega \vec{g}_\omega)\|_{H^{-1}(D)} \le \| f_\omega\|_{H^{-1}(D)} + 
a_{\max}(\omega) \|\vec{g}_\omega\|_{L^2(D)} \, ,
$$
the bound on the $H^1$-norm of $u_\omega$ is a consequence of the 
Lax-Milgram Lemma. 
\end{proof}
Since the solution of \eqref{3.3} is also the solution of the second-order problem \eqref{primal-weak}, we can use the
regularity results established in 
\cite[Prop.~3.1]{CST:2011} and \cite[Thm.~2.1]{TSGU:2012} 
to deduce the regularity of the solution of \eqref{3.3}. We only state the 
result for $d=2$. Similar results can also be proved for $d=3$
and for coefficients that are only piecewise $C^t(\overline{D})$ 
(see \cite{TSGU:2012} for details).
\begin{theorem}
\label{prop:ChScTe}
Let 
$D \subset \mathbb{R}^2$ be a polygon whose largest interior angle is 
$\theta_{\max} \in (0,2\pi)$ and suppose Assumptions A1-A3 are satisfied 
for some $0 < t < 1$.  Then 
$(\vq_\omega,u_\omega) \in H^s(D)^d \times H^{1+s}(D)$,
for all \ $0<s<\min(t,\frac{\pi}{\theta_{\max}})$, $P$-a.s.~in $\omega \in \Omega$, and the following bounds hold:
\begin{equation}\label{eq:mfe3}
\Vert u_\omega \Vert_{H^{1+s}(D)} \ \lesssim \ C_{\mathrm{reg}}(\omega) 
\quad \text{and} \quad
\Vert \vq_\omega  \Vert_{H^s(D)} \ \lesssim \ \Vert a_\omega \Vert_{C^t(\overline{D})}
\left[C_{\mathrm{reg}}(\omega)+\Vert\vec{g}_\omega\Vert_{H^s(D)}\right]\, ,
\end{equation}
where\vspace{-1.5ex} 
\begin{equation*}
\label{C1}
C_{\mathrm{reg}}(\omega):=
\frac{\amax(\omega) \Vert a_\omega \Vert_{C^t(\overline{D})}^2}
  {\amin^4(\omega)}\left[\Vert f_\omega\Vert_{L^2(D)} + \Vert a_\omega\Vert_{\mathcal{C}^t(\overline{D})} \left(\Vert
\vec{g}_\omega\Vert_{H^{s}(D)} + {\Vert u_\Gamma \Vert_{H^{1/2+s}(\partial D)}}\right) \right]
\end{equation*}
Moreover, $\rdiv \vq_\omega \in H^{s}(D)$ and $\| \rdiv \vq_\omega \|_{H^{s}(D)} = 
\| f_\omega \|_{H^{s}(D)}$ for all \ $0\leq s \leq t$.
\end{theorem} 

\begin{proof}
Let us fix $\omega \in \Omega$.
Since $u_\omega$ is also a solution of the primal problem \eqref{primal-weak} 
(cf.~Lemma \ref{equivalence}), 
the regularity for $u_\omega$ and the bound on 
$\Vert u_\omega \Vert_{H^{1+s}(D)}$ follow immediately from 
\cite[Thm.~2.1]{TSGU:2012} provided
${\tilde{f}_\omega \in H^{-1+s}(D)}$. 
To show that $\rdiv(a_\omega\vec{g}_\omega) \in H^{-1+s}(D)$ we can 
use \cite[Lemma A.2]{CST:2011}, i.e. for any $\phi \in C^t(\overline{D})$ 
and $\psi \in H^s(D)$ with $0<s<t<1$,
it follows that
\begin{equation}
\label{lem:CST_A2}
\Vert \phi \psi  \Vert_{H^s({D})}\  \lesssim \ \Vert
\phi   \Vert_{C^t({\overline{D}})}  \Vert  \psi  \Vert_{H^s({D})} \ .
\end{equation}
If we apply this estimate with $\phi = a_\omega$
and $\psi = (g_\omega)_i$, $i=1,\dots,d$, we get
\[
\Vert a_\omega\vec{g}_\omega  \Vert_{H^s(D)} \ 
\lesssim
\ \Vert a_\omega \Vert_{C^t(\oD)} \Vert \vec{g}_\omega\Vert_{H^{s}(D)}
\,.
\]
Since $\rdiv$ is a linear and continuous operator from $H^s(\mathbb{R}^d)^d$ to $H^{-1+s}(\mathbb{R}^d)$ (cf.
\cite[Remark 6.3.14(b)]{Hackbusch}) it follows, as in the proof of \cite[Thm.~2.1]{TSGU:2012}, by a localisation 
argument that 
\[
\Vert \rdiv(a_\omega\vec{g}_\omega)  \Vert_{H^{-1+s}(D)} \ 
\lesssim
\ \Vert a_\omega \Vert_{C^t(\oD)} \Vert \vec{g}_\omega\Vert_{H^{s}(D)}
\,.
\]   
To bound $\Vert \vq_\omega  \Vert_{H^s(D)}$ we use \eqref{uprime} and again \eqref{lem:CST_A2} with $\phi = a_\omega$
and $\psi = \frac{\partial u_\omega}{\partial x_i}+(\vec{g}_\omega)_i$, $i=1,\dots,d$. 
We get
\[
\Vert \vq_\omega  \Vert_{H^s(D)} \ 
\lesssim
\ \Vert a_\omega \Vert_{C^t(\oD)} (\Vert u_\omega\Vert_{H^{1+s}(D)} + \Vert \vec{g}_\omega\Vert_{H^s(D)})\,.
\]
The fact that $\rdiv \vq_\omega \in H^{s}(D)$ and that its 
$H^{s}(D)$-norm is equal to that of $f_\omega$ for all $0\leq s \leq t$ follows from
the second equation in \eqref{3.3} which implies $\rdiv \vq_\omega \equiv_{L^2} f_\omega$.
\end{proof}

\begin{remark}
\label{full-reg}
For convex domains $D\subset \mathbb{R}^d$ and 
for input random fields $a$ that are sufficiently smooth such that $t=1$, e.g.
for the Mat\'ern covariance with $\nu>1$ or for the Gaussian covariance,
it is proved in \cite[Theorem 2.1]{TSGU:2012} that in fact 
$u_\omega \in H^2(D)$ $P$-a.s.~in $\omega \in \Omega$. 
Provided $f$ is also sufficiently smooth, that is $f\in H^1(D)$, then the Darcy
velocity $\vq_\omega \in H^1(D)^d$ and $\rdiv \vq_\omega \in H^1(D)$. 
In that case, the theoretical results below yield optimal error estimates. 
We will not state that explicitly every time.
\end{remark}

\section{Mixed finite element discretisation and error estimates}\label{sec:error}

The mixed finite element discretisation of (\ref{3.3}), for any 
$\omega \in \Omega$, is obtained by
choosing finite dimensional subspaces $\cV_h \subset \cV $ and 
$\cW_h \subset \cW$ and seeking 
 $(\vq_{h,\omega}, u_{h,\omega}) \in \cV_h \times \cW_h$ such that
\begin{equation}
\left.
\begin{array}{rclll}
\hspace*{1.5cm} m_\omega(\vq_{h,\omega},\vv_h) & + \ \ b(\vv_h,u_{h,\omega})
& = & G(\vv_h) &\mathrm{for\ all\ } \vv_h  \in \cV_h, \\[1ex]
b(\vq_{h,\omega},w_h)&  & = & F_\omega(w_h) 
&\mathrm{for\ all\ }
w_h \in \cW_h\ .
\end{array}
\right\}
\label{3.4}
\end{equation}
Here, for simplicity, we restrict our attention to the case
when $\cV_h$ is the lowest order Raviart-Thomas space on simplices 
\cite{RaviartThomas:1977}. 

Let $\cT_h$ denote a family of triangulations (meshes) of $D$ into conforming $d$-simplices $T \in \cT_h$ (i.e.
triangles for $d=2$ and tetrahedra for $d=3$).
We assume that $\cT_h$ has maximum mesh size $h:=\max_{T \in \cT_h} \diam(T)$ and is nondegenerate as $h\rightarrow 0$,
i.e. $\diam(T)/ \rho_T \leq \gamma_0$ for all $T \in \cT_h$, where $\rho_T$ denotes the radius of the largest closed
ball contained in
$\overline{T}$ with a constant $\gamma_0$ independent of $h$.
Let $\cE$ denote the set of all faces of the simplices in $\cT_h$, that is triangle edges ($d=2$) or tetrahedron
faces ($d=3$), respectively, with predescribed unit normal $\vec{\nu}_E$. 
Let $\cE_I$ and $\cE_D$ denote the subsets of $\cE$ consisting of interior faces $E \subset D$ and
boundary faces $E \subset \partial D$.
For each element $T \in \cT_h$ we define the space of shape functions 
\be
\label{3.6}
RT_0(T):= \{ \vv: T \rightarrow \mathbb{R}^d \ \vert \ \vv(\vx) = \vec{\alpha} + \beta \vx,\, \vec{\alpha} \in
\mathbb{R}^d,\, \beta \in \mathbb{R} \}\ ,
\ee
as well as the global space of discontinuous, piecewise $RT_0$ finite element 
functions
\begin{equation}\label{RT-1}
RT_{-1}(\cT_h) :=\{ \vv \ \vert \ \vv \vert_T \in RT_0(T) \ \forall T \in \cT_h\} \ .
\end{equation}
Finally, the lowest order Raviart-Thomas space $\cV_h$ is defined as 
\begin{equation}\label{RT0}
\cV_h := RT_0(\cT_h):= \{\vv \in RT_{-1}(\cT_h) \ \vert \  \, \vv \cdot \vec{\nu}_E|_E \text{ is continuous across } E,\, \forall E \in \cE_I  \} \ .
\end{equation}
Because of the special form of (\ref{3.6}) it is easily shown that for any $\vv \in RT_0(T)$ the normal
component $\vv \cdot \vec{\nu}_E$ is constant on any face $E$ of $T$. 
Moreover, $\vv_h \in \cV_h$ can be completely determined by specifying the constant value of
$\vv_h \cdot \vec{\nu}_E$ for each $E \in \cE_I \cup \cE_D$. 
In addition, we define $\cW_h$ to be the space of piecewise constant
functions on $D$ with respect to the mesh $\cT_h
$.
The pair $(\cV_h , \cW_h)$ enjoy the commuting diagram property 
\cite[p.109]{BBF:2013} which implies in particular that $\rdiv \cV_h = \cW_h$. 

It follows directly from \eqref{infsup} that $b$ is inf-sup stable with the same constant $k_0>0$ on the finite
element space $\mathcal{V}_h \times \mathcal{W}_h$. 
Similarly, $m_\omega$ is coercive on
\begin{equation}
Z_h \ = \ \{ \vv_h \in \cV_h : b(\vv_h , w_h ) = 0 \quad \text{for
  all} \quad w_h \in \cW_h\} \ ,\label{Zh}
\end{equation}
as well with the same coercivity constant $a_{\max}(\omega)^{-1}$ as in \eqref{coercive}.
It follows (analogously to Proposition~\ref{prop2.2}) from standard results on 
(discretised) mixed variational problems that the discrete variational 
problem \eqref{3.4} has a unique solution (see, e.g., 
\cite[Theorem 4.2.3]{BBF:2013}) with finite $p^\text{th}$ moments up to a certain order. 

\begin{proposition}
Under the Assumptions A1--A3, the family of discrete problems \eqref{3.4} has a 
unique solution $(\vq_h,u_h)$ with $\vq_h \in L^p(\Omega,\cV_h)$ and $u_h \in L^p(\Omega,\cW_h)$, for all $1 \le p < r$.
\end{proposition}

\subsection{Velocity approximation}\label{sec:error:flux}

We fix a realisation $\omega \in \Omega$ and compare the solutions of 
\eqref{3.3} and \eqref{3.4}. To this end we use the following standard 
results from the theory of mixed finite elements in 
\cite{BBF:2013,DouglasRoberts:1985,FalkOsborn:1980}.
\begin{proposition}\label{prop2.6}
We have $Z_h \subset Z$, as well as
\begin{eqnarray*}
\Vert \vq_\omega - \vq_{h,\omega} \Vert_{H(\rdiv)} & \le & \left(1+\frac{1}{k_0}\right) \left(1 +
\frac{a_{\max}(\omega)}{a_{\min}(\omega)}\right)  \inf_{\vv_h \in \cV_h} \Vert \vq_\omega - \vv_h \Vert_{H(\rdiv)} \ \
\textnormal{and} \\[1ex]
\Vert \vq_\omega - \vq_{h,\omega} \Vert_{L^2(D)} \; & \le & \left(1 + \frac{a_{\max}(\omega)}{a_{\min}(\omega)}\right) \inf_{\vv_h \in \cV_h} \Vert \vq_\omega -
\vv_h \Vert_{L^2(D)} \ . 
\end{eqnarray*}
\end{proposition}

Now we  extend the classical error estimates for  lowest order Raviart-Thomas interpolation \cite{RaviartThomas:1977} to fractional order spaces. 
Surprisingly there does not seem to be a good source of a proof for this result, although the analogous result for
hexahedral elements in 3D is given in \cite[Lemma 3.3]{BGNR:2006}.
Let $\Pi_h:H^t(D)^d \cap \Hdiv \to \cV_h$ be the  interpolation operator onto the lowest 
order Raviart-Thomas space, as defined, for example, in \cite[\S~2.5]{BBF:2013}.
Note that for $0<t\leq 1$ the space $H^t(D)^d$, $d=2,3$, is continuously embedded in $L^r(D)^d$, for
some $r>2$ (see, e.g. \cite[Chapter 7]{Adams:2003}).
This is sufficient to ensure that $\Pi_h$ is well-defined even when $t$ approaches $0$ (see, e.g.
\cite[\S~2.5]{BBF:2013}).




\begin{lemma}
\label{th:mfe2}
Let $0 < t\leq 1$. 
Then, for any $\vv \in \Hdiv \cap H^t(D)^d$
\begin{equation}
\label{331a}
\Vert \vv - \Pi_h\vv \Vert_{L^2(D)} \; \lesssim \;  h^t  | \vv |_{H^t(D)} + h \Vert \rdiv \vv \|_{L^2(D)}\,.
\end{equation}
Moreover, for any $\vv \in \Hdiv$ with $\rdiv  \vv \in H^t(D)$,
\begin{equation}
\label{331b}
\Vert \rdiv (\vv - \Pi_h\vv) \Vert_{L^2(D)} \; \lesssim  \; h^t \,  \Vert \rdiv \vv \Vert_{H^t(D)}\, . 
\end{equation}
\end{lemma}
%
\begin{proof}
Let $\hat{T}$ be the $d-$dimensional unit simplex  and let $\hat{\Pi}$ denote the 
Raviart-Thomas interpolation operator of lowest order on $\hat{T}$.  
Moreover let $\hat{\vec{v}}$ 
be any function in $H(\rdiv, \hat{T}) \cap H^t(\hat{T})^d$.   
Then, since  $\hat{\Pi}$ preserves constants, we have, 
for any constant vector $\hat{\vec{p}}$, 
\begin{eqnarray}
\Vert \hat{\vec{v}} - \hat{\Pi}\hat{\vec{v}}\Vert_{L^2(\hat{T})}   \ &= & \ 
\Vert (\hat{\vec{v}}-\hat{\vec{p}})  - \hat{\Pi}( \hat{\vec{v}} - \hat{\vec{p}}) \Vert_{L^2(\hat{T})} \nonumber \\
\ & \leq & \ 
\Vert \hat{\vec{v}}-\hat{\vec{p}}\Vert_{L^2(\hat{T})}    \ + \ \Vert  \hat{\Pi}( \hat{\vec{v}} - \hat{\vec{p}}) \Vert_{L^2(\hat{T})}\ . \label{330} 
\end{eqnarray}
Remembering that the degrees of freedom of $\hat{\Pi}\hat{\vec{v}}$ 
are the integrals of the normal component of $\hat{\vec{v}}$ on each of the faces of $\hat{T}$, we can argue as in Lemma 3.15 and equation (3.39)  of \cite{Hi:02} to obtain 
\begin{eqnarray}
\Vert  \hat{\Pi}( \hat{\vec{v}} - \hat{\vec{p}}) \Vert_{L^2(\hat{T})}\ 
& \lesssim  & \ \Vert  \hat{\vec{v}} - \hat{\vec{p}} \Vert_{H^t(\hat{T})} 
\ + \Vert  \rdiv(\hat{\vec{v}} - \hat{\vec{p}}) \Vert_{L^2(\hat{T})},   \nonumber \\
& = &  \ \Vert  \hat{\vec{v}} - \hat{\vec{p}} \Vert_{L_2(\hat{T})}  \ + \ \vert \hat{\vec{v}} \vert_{H^t(\hat{T})}   
\ + \ \Vert  \rdiv \, \hat{\vec{v}} \Vert_{L^2(\hat{T})}\ , 
\label{331}
\end{eqnarray}
for any $t>0$. 
Inserting this into \eqref{330} we conclude that 
 \begin{eqnarray}
 \Vert \hat{\vec{v}} - \hat{\Pi}\hat{\vec{v}}\Vert_{L^2(\hat{T})}
& \lesssim  &  
\Vert  \hat{\vec{v}} - \hat{\vec{p}} \Vert_{L_2(\hat{T})}\ + \ 
\ \vert  \hat{\vec{v}} \vert_{H^t(\hat{T})} 
\ + \Vert  \rdiv \, \hat{\vec{v}} \Vert_{L^2(\hat{T})}\ . 
\label{332}\end{eqnarray}

Now as $\hat{\vec{p}}$ is arbitrary we can use the 
Bramble-Hilbert Lemma in fractional order spaces 
(e.g. \cite[Prop. 6.1]{DuSc:80}) to estimate the $L_2$ term on the right hand side of 
\eqref{332}, obtaining, in the end, 
\begin{eqnarray}
 \Vert \hat{\vec{v}} - \hat{\Pi}\hat{\vec{v}}\Vert_{L^2(\hat{T})}
& \lesssim  & \ 
\ \vert  \hat{\vec{v}} \vert_{H^t(\hat{T})} 
\ + \Vert  \rdiv \, \hat{\vec{v}} \Vert_{L^2(\hat{T})}\ . 
\label{336}
\end{eqnarray}

Now for any simplex $T \in \cT_h$, take any function $\vec{v} \in
H(\rdiv,T)\cap H^t(T)^d$.  Let $F_T: \hat{T} \rightarrow T$ be the usual affine 
map  with (constant) Jacobian $DF_T$ and set $J_T = \mathrm{det}(DF_T)$. The 
Piola transform of $\vec{v}$ is 
$\hat{\vec{v}} = J_T (DF_T)^{-1} (\vec{v}\circ F_T)$,  and simple scaling arguments 
(see, e.g. \cite{BGNR:2006}) show that 
$\vert \hat{\vec{v}} \vert_{H^t(\hat{T})} \sim h_T^{1/2+t} 
\vert \vec{v} \vert_{H^t(T)}$. Moreover $\hat{\Pi} \hat{\vec {v}} = 
\widehat{\Pi_T\vec{v}}$, where $\Pi_T$ is the Raviart-Thomas interpolation operator  on $T$.     
Using  these results in \eqref{336} we obtain 
\begin{eqnarray*}
 \Vert {\vec{v}} - {\Pi}_T{\vec{v}}\Vert_{L^2({T})}
& \lesssim  &  
\ h_T^t \vert  {\vec{v}} \vert_{H^t({T})} 
\ + h_T \Vert  \rdiv \, {\vec{v}} \Vert_{L^2({T})}\ . 
\end{eqnarray*}
Then, squaring and summing over all $T \in \cT_h$ we obtain 
\eqref{331a}.

The second bound is simpler since  
$\rdiv( \Pi_h \cdot) = P_h (\rdiv \cdot)$,  where $P_h$ is  the $L^2$-orthogonal 
projection $P_h:\cW \to \cW_h$  (see e.g. \cite[Prop. 2.5.2]{BBF:2013}). 
Then we have 
$$
\Vert \rdiv (\vv - \Pi_h\vv) \Vert_{L^2(D)}  \ = \ \Vert (I  - P_h) \rdiv \vv 
\Vert_{L^2(D)}\ , 
$$
and \eqref{331b} follows directly by standard polynomial approximation results in fractional order spaces (see, e.g.  \cite{ScottZhang:1990}). 
\end{proof}

The next theorem now follows immediately by combining Proposition~\ref{prop2.6} and Lemma~\ref{th:mfe2} with
Theorem~\ref{prop:ChScTe}. 

\begin{theorem} 
\label{cor:mfe5}
Let the assumptions of Theorem \ref{prop:ChScTe} hold. 
Then we have $P$-a.s. in $\omega \in \Omega$ and for all 
\ $0<s<\min(t,\frac{\pi}{\theta_{\max}})<1$ that 
\vspace{-1ex}
\begin{equation}
\label{eq:mfe5}
\Vert \vq_\omega - \vq_{h,\omega}  \Vert_{L^2(D)} \lesssim  C_q(\omega) h^s
\quad \text{ and } \quad
\Vert \vq_\omega - \vq_{h,\omega}  \Vert_{\Hdiv} \lesssim  \left(C_q(\omega)+\frac{\amax(\omega)}{\amin(\omega)}\Vert
f_\omega\Vert_{H^t(D)}\right) h^s 
\end{equation}
where\vspace{-1ex}
\begin{equation}
\label{eq:mfe5b}
C_q(\omega) \; := \; \frac{\amax(\omega)\Vert a_\omega \Vert_{C^t(\overline{D})}}{\amin(\omega)}
\,\left(C_{\mathrm{reg}}(\omega)+\Vert\vec{g}_\omega\Vert_{H^t(D)}\right)\,  .
\end{equation}
\end{theorem}
Due to Assumptions A1--A3 the following corollary is a consequence of H\"older's inequality:
Recall that A1--A2 imply $a_{\max}$, $1/a_{\min}$, and $\Vert a_\omega\Vert_{C^t(\overline{D})}$ are in
$L^p(\Omega)$, for all $p \in [1,\infty)$.
Thus, any product of these quantities with a finite number of factors is again in $L^{p}(\Omega)$, for all $p
\in [1,\infty)$, due to H\"{o}lder's inequality.
Applying H\"{o}lder's inequality again together with Assumption A3 we conclude that $C_{\mathrm{reg}} \in
L^p(\Omega)$ and $C_{q} \in L^p(\Omega)$, for all $p<r$.
\begin{corollary}
\label{cor:velocity}
Let $t \in (0,1)$ and $r \in [1,\infty)$ be the parameters in Assumptions A1--A3. 
Then, under the assumptions of Theorem \ref{prop:ChScTe} and with $0<s<\min(t,\frac{\pi}{\theta_{\max}})<1$ 
we have
\[
\Vert \vq - \vq_h  \Vert_{L^p(\Omega, L^2(D))}   \lesssim    h^s \quad \text{and} \quad
\Vert \vq - \vq_{h}  \Vert_{L^p(\Omega, \Hdiv)}   \lesssim    h^s\,, \quad
\text{for all} \ \ p < r.
\]
\end{corollary}

\subsection{Pressure approximation}\label{sec:error:press}

Recall that realisations $a_\omega$ of the diffusion
coefficient are in general only H\"{o}lder continuous with exponent
$0<t\leq 1$  (cf. Assumption A2).
Standard arguments from \cite{BBF:2013} would lead only to
a pressure approximation error (for each sample $\omega \in \Omega$) of order 
$h^s$, for all $s<t$. We will show that the approximation error is actually of 
order $h$.
To prove this estimate we first prove an auxiliary result, based on a duality argument.
This estimates the difference between the mixed finite element approximation and the
$L^2$-orthogonal projection of the exact pressure in
$\mathcal{W}_h$. This is similar to the arguments in
\cite{DouglasRoberts:1985,FalkOsborn:1980} but surprisingly we have
not been able to find the error bound in \eqref{ubound} in the literature.

Let $P_h$ denote the $L^2$-orthogonal projection introduced in the proof of Lemma~\ref{th:mfe2}.
Let $w \in \cW$ and note that since $\rdiv \cV_h = \cW_h$, we also have 
\begin{equation}
\label{eq:bortho}
b(\vv_h,P_h w) = -(\rdiv \vv_h,P_h w)_{L^2(D)} =  -(\rdiv \vv_h,w)_{L^2(D)} =
b(\vv_h,w), \qquad \textnormal{for all } \vv_h \in \cV_h\ .
\end{equation}
Before we move on, recall the following classical result (see e.g. \cite{RaviartThomas:1977})
\begin{equation}
\label{lem:L2approx}
\|w - P_h w \|_{L^2(D)} \lesssim h \|w\|_{H^1(D)} \qquad w \in H^1(D) \ .
\end{equation}
\begin{lemma}
\label{th:mfe6}
Under the assumptions of Theorem \ref{prop:ChScTe}, for $P$-a.s. $\omega \in \Omega$
\begin{equation}
\label{eq:dualbound}
\Vert P_h u_\omega - u_{h,\omega} \Vert_{L^2(D)} \;\lesssim\;
\underbrace{\frac{a_{\max}(\omega)}{a^2_{\min}(\omega)} \left(\frac{a_{\max}(\omega) \Vert a_\omega
\Vert^3_{C^t(\overline{D})}}{a^4_{\min}(\omega)} \;C_q(\omega) \,+\, \|f_\omega\|_{H^{t}(D)}
\right) }_{=: C_u(\omega)} \; h^{2s} \, . 
\end{equation}
\end{lemma}

\begin{proof}
Following the argument in \cite[p. 432]{BBF:2013}, 
consider the dual mixed problem to find $(\vz,\phi) \in \cV \times \cW$, s.t.
\begin{equation}
\begin{array}{rclll}
\hspace*{1.5cm} m_\omega(\vv,\vz) & + \ \ b(\vv,\phi)
& = & 0\ ,&\mathrm{for\ all\ } \vv  \in \cV, \\[1ex]
b(\vz,w)&  & = & (P_h u-u_h,w)_{L^2(D)} \ , &\mathrm{for\ all\ }
w \in \cW\; .
\end{array}
\label{standard-dual}
\end{equation}
Recall that $P_h u-u_h$ and the dual solution $(\vz,\phi)$ depend on the sample $\omega \in \Omega$ but we shall omit this relation in what follows.
In the associated discrete dual problem we shall use the subscript $h$
and replace $\cV$ by $\cV_h$ and $\cW$ by $\cW_h$, respectively. In
particular, we denote by $(\vz_h,\phi_h) \in \cV_h \times \cW_h$ the
mixed finite element solution of \eqref{standard-dual}.

Due to \eqref{eq:bortho}, it follows from the second equation in the discrete version of \eqref{standard-dual} that
\begin{equation}
\label{extra}
\Vert P_h u -u_h \Vert_{L^2(D)}^2 
=
b(\vz_h,P_h u-u_h) 
= 
b(\vz_h,u-u_h) 
=
m_\omega(\vq_h-\vq,\vz_h)
\end{equation}
where in the last step we used \eqref{3.3} and \eqref{3.4} with test function
$\vv=\vv_h=\vz_h$.
Now, using the first equation in \eqref{standard-dual} with test
function $\vv = \vq_h-\vq$ we further deduce that
\begin{equation*}
\begin{split}
\Vert P_h u -u_h \Vert_{L^2(D)}^2  
=
m_\omega(\vq_h-\vq,\vz_h-\vz) + b(\vq-\vq_h,\phi) 
=
m_\omega(\vq_h-\vq,\vz_h-\vz) + b(\vq-\vq_h,\phi-P_h \phi) \ .
\end{split}
\end{equation*}
In the final step we have simply used
the second equations in \eqref{3.3} and \eqref{3.4} with test function
$w=w_h=P_h \phi$, respectively.
A simple application of the Cauchy-Schwarz inequality leads to
\begin{equation}
\label{dualbound}
\Vert P_h u -u_h \Vert_{L^2(D)}^2 \lesssim
 \frac{1}{a_{\min}(\omega)} \Vert \vq-\vq_h \Vert_{L^2(D)} \Vert
\vz-\vz_h \Vert_{L^2(D)} + \Vert \rdiv(\vq-\vq_h) \Vert_{L^2(D)} \Vert
\phi - P_h \phi \Vert_{L^2(D)}  \,.
\end{equation}
Since $m_\omega$ is symmetric, \eqref{standard-dual} is a special case
of \eqref{3.3} with data $\vec{g}_{dual} \equiv \vec{0}$, 
$f_{dual} = u_h-P_h u \in L^2(D)$ and $u_{\Gamma,dual} \equiv 0$.
This allows us to bound the error $\|\vz - \vz_h\|_{L^2(D)}$ using Theorem~\ref{cor:mfe5}.
The bound \eqref{eq:mfe5} (applied to $\vq=\vz$) gives
\[
\|\vz - \vz_h\|_{L^2(D)} \lesssim C_q(\omega)\, h^s =  \frac{a^2_{\max}(\omega) \Vert a_\omega
\Vert^3_{C^t(\overline{D})}}{a^5_{\min}(\omega)} \Vert P_h u
-u_h \Vert_{L^2(D)} \; h^s \ .
\]
Moreover, it follows from \eqref{lem:L2approx} (applied to $w =
\phi$) and \eqref{eq:H1bound} that
\begin{eqnarray*}
\Vert \phi - P_h \phi \Vert_{L^2(D)} 
& \lesssim &
\Vert \phi \Vert_{H^1(D)} \; h 
\lesssim 
\frac{1}{a_{\min}(\omega)} \Vert P_h u
-u_h \Vert_{L^2(D)} \; h \,.
\end{eqnarray*}
Using these bounds together with Theorem~\ref{cor:mfe5} (applied to $\vq$) in \eqref{dualbound} and 
dividing the result by $\Vert P_h u -u_h \Vert_{L^2(D)}$ we obtain the final bound in \eqref{eq:dualbound}.

\end{proof}

Using \eqref{lem:L2approx} and Lemma~\ref{th:mfe6} we can now establish
the improved convergence order for the pressure error for $0<t<1$.

\begin{theorem}
\label{th:super-p0}
Let the assumptions of Theorem \ref{prop:ChScTe} hold. Then, for all \ $0<s<\min(t,\frac{\pi}{\theta_{\max}})<1$, 
\begin{equation}
\label{ubound}
\Vert u_\omega - u_{h,\omega}  \Vert_{L^2(D)} \; \lesssim \; C_u(\omega)\; h^{\min(2s,1)}
\end{equation}
where $C_u$ is defined in \eqref{eq:dualbound}.
Moreover, $\Vert u - u_h  \Vert_{L^p(\Omega, L^2(D))} \lesssim h^{\min(2s,1)}$, for all $p < r$.
\end{theorem}
\begin{proof}
It follows from \eqref{lem:L2approx} applied to $w = u_\omega$ and from Theorem \ref{prop:ChScTe}
that 
\[
\Vert u_\omega - P_h u_\omega \Vert_{L^2(D)} 
\lesssim 
\Vert u_\omega \Vert_{H^1(D)}\; h
\leq
\Vert u_\omega \Vert_{H^{1+s}(D)}\; h
\lesssim
C_\text{reg}(\omega)\; h
\le
C_u(\omega)\; h
\]
with $C_u(\omega)$ as defined in \eqref{eq:dualbound}.
Combining this bound with Lemma~\ref{th:mfe6} the bound \eqref{ubound} follows immediately via the triangle inequality.
Due to Assumptions A1--A3 the bound on the moments is then again a consequence of H\"older's inequality.
\end{proof}

\subsection{Linear velocity functionals}\label{sec:linfluxfunc}

Let $\mathcal{M}:\cV \to \mathbb{R}$ be a continuous linear functional of the Darcy
velocity $\vq$. An important example is the effective permeability, which 
in a rectangular flow cell reduces simply to 
\begin{equation}
\label{keff-func}
\mathcal{M}(\vq) = \frac{1}{|D|} \int_D q_1 \dx
\end{equation}
(see Section~\ref{sec:numerics} for details).
Another example is the average normal flux through some part of the 
boundary. Our goal in this section is to estimate the FE approximation error 
$|\mathcal{M}(\vq)-\mathcal{M}(\vq_h)|$.

Let $\omega \in \Omega$ be again fixed.
We will again use a duality argument and so we introduce the following auxiliary problem: 
Find $(\vz_\omega,\phi_\omega) \in \cV \times \cW$, s.t.
\begin{equation}
\left.
\begin{array}{rcll}
\hspace*{1.5cm} m_\omega(\vv,\vz_\omega) & + \ \ b(\vv,\phi_\omega)
& = & \mathcal{M}(\vv)\ ,
\; \; \mathrm{for\ all\ } \vv  \in \cV, \\[1ex]
b(\vz_\omega,w)&  & = & 0 \ , \; \; \; \; \; \;\;\;\;
\mathrm{for\ all\ }
w \in \cW\ .
\end{array}
\right\}
\label{standard-dual-flux}
\end{equation}
We consider a specific class of linear functionals $\mathcal{M} \in \cV'$.
Let $\vv \in \Hdiv$, then by the Riesz Representation Theorem
\[
\mathcal{M}(\vv) = \int_D \vec{\psi}\cdot \vv \dx + \int_D \rdiv \vec{\psi} \, \rdiv \vv \dx 
\]
for some $\vec{\psi} \in \Hdiv$.
\begin{description}
\item[A4.] We assume that $\mathcal{M} \in \cV'$ has a Riesz
  representor $\vec{\psi} \in H^t(D)^d$ with $\rdiv \vec{\psi} \in H^{1+t}(D)$, for some $0<t\leq 1$.
\end{description}
Note that under this assumption we have, due to Green's formula, 
\[
\int_D \rdiv \vec{\psi} \, \rdiv \vv \dx
=
\int_{\partial D} \vv \cdot \vec{\nu} \,\rdiv \vec{\psi}ds - \int_D \vv \cdot \nabla (\rdiv \vec{\psi}) \dx
\]
and hence we can write 
\begin{equation}
\label{linfunc}
  \mathcal{M}(\vv) = \int_D \vv \cdot (\vec{\psi}-\nabla(\rdiv \vec{\psi}) ) \dx + \int_{\partial D} \vv \cdot
\vec{\nu} \, \rdiv \vec{\psi} \, ds \ .
\end{equation}
For the functional in \eqref{keff-func}, i.e. the effective permeability in a rectangular flow cell, the Riesz
representor is $\vec{\psi} = (1/|D|,0)^\top$. 
For this choice of $\vec{\psi}$ Assumption A4 is clearly satisfied.

We now have the following result.
\begin{theorem}\label{th:mfe7}
Let Assumption A4 and the assumptions of Theorem \ref{prop:ChScTe} hold. Then,
\[
\|\mathcal{M}(\vq)-\mathcal{M}(\vq_{h})\|_{L^p(\Omega)} \lesssim h^{2s},
\]
for all $0<s<\min(t,\frac{\pi}{\theta_{\max}})<1$ and all $p < r$.
\end{theorem}
\begin{proof}
Omitting again the dependence on $\omega \in \Omega$, let 
$(\vz,\phi) \in \cV \times \cW$ be the solution to the dual mixed
problem \eqref{standard-dual-flux}, and let $(\vz_h,\phi_h) \in \cV_h
\times \cW_h$ be the corresponding mixed FE solution.

Note first that the second equation in the discrete version of \eqref{standard-dual-flux} 
implies $\rdiv \vz_h \equiv 0$ on all of $D$. Thus, by subtracting the 
first equation in \eqref{3.3} from that
in \eqref{3.4} with test functions $\vv=\vv_h=\vz_h$, 
respectively, this also implies that
\begin{equation}
\label{velfunc_aux}
m_\omega(\vq-\vq_h,\vz_h) 
=
-b(\vz_h,u-u_h)
= 0 \, .
\end{equation}
Now, choosing $\vv=\vq-\vq_h$ in the first equation of 
\eqref{standard-dual-flux} and using \eqref{velfunc_aux}, as well as 
the bilinearity of $m_\omega$, we obtain
\begin{equation*}
\mathcal{M}(\vq-\vq_h)
=
m_\omega(\vq-\vq_h,\vz) + b(\vq-\vq_h,\phi)
=
m_\omega(\vq-\vq_h,\vz-\vz_h) +  b(\vq-\vq_h,\phi - P_h \phi) \, .
\end{equation*}
In the last step  we used (as in the proof of Lemma~\ref{th:mfe6})
that $b(\vq-\vq_h,P_h \phi)=0$. Thus, by the Cauchy-Schwarz inequality we finally get
\begin{equation}
\label{funcbound}
|\mathcal{M}(\vq-\vq_h)| \;\le \;
\frac{1}{a_{\min}(\omega)} \Vert\vq-\vq_h\Vert_{L^2(D)}
\Vert\vz-\vz_h\Vert_{L^2(D)} + \Vert\rdiv(\vq-\vq_h)\Vert_{L^2(D)}
\Vert \phi-P_h\phi\Vert_{L^2(D)} \,.
\end{equation}
Note that \eqref{standard-dual-flux} is a special case of \eqref{3.3} with data
$\vec{g} = \vec{\psi}-\nabla(\rdiv \vec{\psi}) \in H^t(D)^d$ (thanks to Assumption A4), $f \equiv 0$ and $u_\Gamma
=-\rdiv \vec{\psi} \in H^{1/2+t}(D)$ (due to the standard Trace Theorem).
It follows again as in the proof of Lemma \ref{th:mfe6}, by applying \eqref{lem:L2approx} and Lemma~\ref{equivalence},
and Theorem \ref{cor:mfe5} to the dual problem \eqref{standard-dual-flux} that
\begin{eqnarray*}
\|\vz - \vz_h\|_{L^2(D)} & \lesssim & \frac{[\max(a_{\max}(\omega), \Vert a_\omega
\Vert_{C^t(\overline{D})})]^6}{a^5_{\min}(\omega)} \left( \|\vec{\psi}-\nabla(\rdiv \vec{\psi})\|_{H^s(D)} +
\|\rdiv \vec{\psi}\|_{H^{1/2+s}(\partial D)} \right) h^s \,, \\ 
\Vert \phi - P_h \phi \Vert_{L^2(D)} & \lesssim &
\left(\frac{a_{\max}(\omega)}{a_{\min}(\omega)} \Vert \vec{\psi}-\nabla \rdiv \vec{\psi} \Vert_{L^2(D)}
+ | \rdiv \vec{\psi} |_{H^{1/2}(\partial D)} \right) h \,.
\end{eqnarray*}
Substituting these two bounds in \eqref{funcbound} and using Theorem~\ref{cor:mfe5} we obtain 
\[
|\mathcal{M}(\vq_\omega-\vq_{h,\omega})| \;\lesssim\; 
C_{\mathcal{M}}(\omega) 
\left(\Vert\vec{\psi} - \nabla (\rdiv \vec{\psi}) \Vert_{H^t(D)} + \Vert \rdiv \vec{\psi} \Vert_{H^{1/2+t}(\partial
D)}\right) h^{2s}\,,
\]
for some constant $C_{\mathcal{M}}(\omega)$, that (as a function of $\omega$) is a random variable 
$C_{\mathcal{M}} \in L^q(\Omega)$, for all $q<r$, due to Assumptions A1--A3 and H\"older's 
inequality. The result then follows.
\end{proof}

\begin{remark}
The treatment of pressure functionals $\mathcal{L}=
\mathcal{L}(u)$ is similar. The dual problem in that case has
$G_\omega \equiv 0$
and $F_\omega \equiv \mathcal{L}$. 
For ideas on how to generalise to Fr\'{e}chet differentiable, nonlinear functionals see
\cite[\S~3.2]{TSGU:2012}, where this is explained in the standard finite element case. 
\end{remark}

\subsection{Hybridisation and recovered pressure approximation}\label{sec:err:P1}
The Galerkin matrix associated with the mixed formulation \eqref{3.4} is indefinite.
This is a potential problem for both direct and iterative solvers and was considered a major drawback of mixed methods.
\textit{Hybridisation} overcomes this problem by introducing (additional) Lagrange multipliers.
After block elimination of pressures and velocities, the system for the multipliers is symmetric positive definite and much smaller than the saddle point
system associated with \eqref{3.4}.
It can be solved efficiently by multigrid methods (see, e.g. \cite{Chen:1996}).
Crucially, the mixed velocity and pressure approximation can then be obtained by local post-processing.

For lowest order Raviart-Thomas elements the additional Lagrange multipliers are piecewise constant along interelement
boundaries in the finite element triangulation.
They are used in conjunction with the space $RT_{-1}(\cT_h)$ of discontinuous, piecewise $RT_0$
functions (w.r.t. to $\cT_h$) for the velocity approximation defined in \eqref{RT-1}, to enforce the continuity of the
normal component of the velocities across interelement boundaries in a weak sense.
This approach was originally proposed in \cite{FDV:1965} as an efficient implementation technique for mixed methods in
linear elasticity.
However, in \cite{ArnoldBrezzi:1985} it was proved that the Lagrange multipliers contain extra information and can be used to construct a more accurate piecewise linear recovered pressure approximation. In the remainder of this section, we analyse the FE error of the recovered pressure approximation. This is a novel application of this method in the case of random coefficients. In addition, the diffusion coefficient has limited regularity which requires working in non-integer order spaces.

More recently, the Lagrange multipliers have been characterised as the solution of a variational problem in which
velocity and pressure do not appear, leading to a new approach to the error analysis of hybridised mixed methods which
gives error estimates for the Lagrange multipliers without using error estimates for the other variables
\cite{Chen:1996,CockburnGopalakrishnan:2005}.
In our analysis here, however, we follow the traditional approach in \cite{ArnoldBrezzi:1985}.
We mention that in \cite{ArnoldBrezzi:1985} only the case $D \subset \mathbb{R}^2$ and Raviart-Thomas elements of even
orders $k$ have been described in detail, but the construction and the analysis extend also to three space dimensions
for lowest order elements ($k=0$). 
The key is again to use the results in \cite{ArnoldBrezzi:1985} for a fixed realisation $\omega \in \Omega$, keeping
track of the precise dependence of the constants on $\omega$. 

Let $\omega \in \Omega$ be fixed and let $M_{-1}(\cE)$ denote the
space of all piecewise constant functions on the skeleton $\cE = \cE_I
\cup \cE_D$. 
The space of all functions in $M_{-1}(\cE)$
that vanish on boundary faces $E \in \cE_D$ is denoted $M_{-1}(\cE_I)$. 
We modify the standard Raviart-Thomas mixed Galerkin equations \eqref{3.4} by using velocity functions $\underline{\vq}_{h,\omega} \in RT_{-1}(\cT_h)$ together with multipliers $\lambda_{h,\omega} \in M_{-1}(\cE_I)$ and pressure functions $\underline{u}_{h,\omega} \in \cW_h$. 
In addition we introduce the bilinear form $b_T(\vv,w):=- (\rdiv \vv, w)_{L^2(T)}$.
We then seek $(\underline{\vq}_{h,\omega},\underline{u}_{h,\omega},\lambda_{h,\omega}) \in RT_{-1}(\cT_h) \times \cW_h \times M_{-1}(\cE_I)$ such that
\begin{equation}
\left.
\hspace*{-2cm}
\begin{array}{rlll}
\hspace*{1.45cm}m_\omega(\underline{\vq}_{h,\omega},\vv_h) +  \sum\limits_{T \in \cT_h} b_T(\vv_h,\underline{u}_{h,\omega}) 
+ \sum\limits_{E \in \cE_I} (\lambda_{h,\omega}, \vv_h \cdot
\vec{\nu}_E)_{L^2(E)} \!\!
& = & \! G_\omega(\vv_h) & \forall \, \vv_h  \in RT_{-1}(\cT_h), \\[1ex]
\sum\limits_{T \in \cT_h} b_T(\underline{\vq}_{h,\omega},w_h)  \!\! & = & \! F_\omega(\vv_h) 
& \forall \,
w_h \in \cW_h, \\[1ex]
\sum\limits_{E \in \cE_I} (\mu_h, \underline{\vq}_{h,\omega} \cdot \vec{\nu}_E)_{L^2(E)}  \!\! & = & \! 0
& \forall \, \mu_h \in M_{-1}(\cE_I) \ .
\end{array}
\right\} \hspace{-0.1cm}
\label{P1-system}
\end{equation}

By construction it is clear that the hybridised mixed system \eqref{P1-system} has a unique solution $(\underline{\vq}_{h,\omega},\underline{u}_{h,\omega},\lambda_{h,\omega})$. 
Moreover, we have $\underline{\vq}_{h,\omega} = \vq_{h,\omega}$ and $\underline{u}_{h,\omega} = u_{h,\omega}$ where
$(\vq_{h,\omega},u_{h,\omega})$ denotes the solution of the standard Raviart-Thomas mixed Galerkin equations
\eqref{3.4}. 
Hence our analysis of the mixed formulation in Sections~\ref{sec:error:flux}--\ref{sec:linfluxfunc}
carries over to the hybridised formulation and we drop the lower bars in \eqref{P1-system} in the rest of this section.
It only remains to analyse the convergence of the piecewise linear recovered pressure approximation, {which we define below.}

Before defining the {recovered} approximation, we derive bounds for the difference between 
$\lambda_{h,\omega}$ and the trace of $u_\omega$ on the interior 
element boundaries $\cE_I$. 
To this end we define the following {inner product
\begin{equation}
\label{ip:M-1}
( \mu_h, \rho_h)_{0,h} 
:=
\sum_{E \in \cE} (\mu_h,\rho_h)_{L^2(E)} 
\end{equation}
and the corresponding broken norm $\vert \mu_h\vert_{0,h}:=(\mu_h,\mu_h)_{0,h}^{1/2}$ on $M_{-1}(\cE)$.}
Since $u_\omega \in H_0^{1}(D)$, we know that $u_\omega \vert_{\cE} \in L^2(\cE)$ and we denote by
$P_h^\cE$ the orthogonal projection from $L^2(\cE)$ onto $M_{-1}(\cE)$ with respect to the inner product \eqref{ip:M-1}. 
We then have the following approximation result.
\begin{lemma}
\label{lem:M-1}
For every element $T \in \cT_h$ and every edge $E$ of $T$
\begin{equation}
\Vert \lambda_{h,\omega} - P_h^\cE u_\omega \Vert_{L^2(E)} \lesssim \frac{1}{a_{\min}(\omega)} h_T^{1/2} \Vert \vq_\omega - \vq_{h,\omega} \Vert_{L^2(T)}
+ h_T^{-1/2} \Vert u_{h,\omega} - P_h u_\omega \Vert_{L^2(T)} \, ,
\end{equation}
where $h_T:= \diam(T)$ and $P_h$ is the $L^2$-orthogonal projection from $\cW$ to $\cW_h$ defined in Section~\ref{sec:error:press}.
\end{lemma}

\begin{proof}
We follow the proof of \cite[Theorem 1.4]{ArnoldBrezzi:1985} and extend it to {the stochastic case and to} $d=3$ space dimensions.
Consider an element $T \in \cT_h$ and a face $E \subset \partial T$.
Let $(\vq_{h,\omega},u_{h,\omega},\lambda_{h,\omega})$ denote the solution of \eqref{P1-system} and $(\vq_\omega,u_\omega)$ the solution of \eqref{3.3}, respectively. We omit the dependence on $\omega$ in the proof.

Since a function $\vv_h \in RT_0(T)$ is uniquely determined by the (constant) value $\vv_h \cdot \vec{\nu}_E$ on $E_i \subset \partial T$, $i=1,\dots,d+1$, it is clear that there exists a unique $\underline{\vec{\delta}}_h \in RT_0(T)$ such that
\begin{equation}
\left.
\begin{array}{rcll}
\underline{\vec{\delta}}_h \cdot \vec{\nu}_E &=& \lambda_h - P_h^\cE u& \text{on } E,\\
\underline{\vec{\delta}}_h \cdot \vec{\nu}_{E'} &=& 0 &\text{on } E' \in \partial T \setminus E \ .
\end{array} 
\right\}
\label{testf}
\end{equation}  
Moreover, since $\underline{\vec{\delta}}_h \in RT_0(T)$, we have
\begin{equation}
\label{scaling}
\Vert \rdiv \underline{\vec{\delta}}_h\Vert_{L^2(T)} \;+ \;h_T^{-1} \Vert 
\underline{\vec{\delta}}_h\Vert_{L^2(T)} \;\lesssim\; 
h_T^{-\frac{1}{2}}\Vert \lambda_h - P_h^\cE u \Vert_{L^2(E)} \, .
\end{equation}
An elementary calculation shows that this bound holds on the reference element $\widehat{T}$. 
A simple scaling argument then gives \eqref{scaling}.
As in the proof of Lemma \ref{equivalence}, 
a local application of Green's formula gives
\[
\int_T (\vec{g}-a^{-1}\vq) \cdot \underline{\vec{\delta}}_h \dx + \int_T u \, \rdiv \underline{\vec{\delta}}_h \dx = \int_E u(\lambda_h - P_h^\cE u) \, \text{d}E \, ,
\]
which, together with the definitions of the projections $P_h$ and $P_h^\cE$, leads to
\begin{equation}
\label{e2}
\int_T a^{-1} \vq \cdot \underline{\vec{\delta}}_h \dx - \int_T P_h u \, \rdiv \underline{\vec{\delta}}_h \dx + \int_E P_h^\cE u(\lambda_h - P_h^\cE u) \, \text{d}E = \int_T \vec{g} \cdot \underline{\vec{\delta}}_h \dx
\end{equation}
Now, subtracting this from the first equation of \eqref{P1-system} with test function
$$
\vv_h = \underline{\vec{\delta}}_h \text{ in } T \quad \text{and} \quad \vv_h = \vec{0} \text{ in } \Omega \setminus T,
$$
from \eqref{e2} and using the scaling argument \eqref{scaling}, we get
\begin{equation*}
\begin{split}
\Vert \lambda_h - P_h^\cE u \Vert_{L^2(E)}^2 
&=
\int_T a^{-1} (\vq - \vq_h) \cdot \underline{\vec{\delta}}_h  \dx - \int_T (P_h u - u_h) \rdiv \underline{\vec{\delta}}_h \dx \\
&\leq
\frac{1}{a_{\min}(\omega)} \Vert \vq - \vq_h \Vert_{L^2(T)} \Vert \underline{\vec{\delta}}_h \Vert_{L^2(T)}
+
\Vert u_h-P_h u\Vert_{L^2(T)} \Vert \rdiv \underline{\vec{\delta}}_h\Vert_{L^2(T)}\\
&\lesssim
\left(\frac{1}{a_{\min}(\omega)} h_T^{1/2}\Vert \vq - \vq_h \Vert_{L^2(T)} + h_T^{-1/2} \Vert u_h-P_h u \Vert_{L^2(T)}\right) \Vert \lambda_h - P_h^\cE u \Vert_{L^2(E)}\,.
\end{split}
\end{equation*}
Dividing both  sides of the estimate by $\Vert \lambda_h - P_h^\cE u \Vert_{L^2(E)}$ gives the desired result.
\end{proof}

We now define an approximation $\widetilde{u}_{h,\omega} \in \cW_{1;h}$, where 
$\cW_{1;h} \subset L^2(D)$ denotes the Crouzeix-Raviart space of all (discontinuous) 
piecewise linear 
functions that have continuous averages across all element faces $E \in \mathcal{E}_I$
(cf.~\cite[\S~1.2.6]{ErnGuermond:2004}). The construction is based on the following 
observation (cf.~\cite[Lemma 2.1]{ArnoldBrezzi:1985}, the proof of which 
extends in a straight forward way to $d=3$).

\begin{lemma}
\label{lem:pwlin}
Let $w \in L^2(\mathcal{E})$ and $T \in \mathcal{T}_h$. Then there exists a unique
linear function $\chi_T = \chi_T(w)$ on~$T$ such that
\begin{equation}
\label{cond}
\int_E \chi_T \, \text{d}E = \int_E w \, \text{d}E
\end{equation}
for all faces $E \subset \partial T$. Moreover, 
\begin{equation}
\label{est}
\Vert \chi_T \Vert_{L^2(T)} \; \lesssim \; h_T^{1/2} 
\sum_{E \subset \partial T} \Vert w \Vert_{L^2(E)} \, .
\end{equation}
The hidden constant depends only on the minimum angle of \ $T$.
\end{lemma}
We define the recovered pressure
approximation $\widetilde{u}_{h,\omega} \in \cW_{1;h}$ elementwise by
\begin{equation}
\label{def:P1}
\widetilde{u}_{h,\omega}|_T := \chi_T\left(\lambda_{h,\omega}\right)\,, \quad 
\text{for all} \ \ T \in \mathcal{T}_h\,.
\end{equation}
Lemma \ref{lem:pwlin} ensures that $\widetilde{u}_{h,\omega} \in \cW_{1;h}$ is well-defined 
and unique. The value of 
the multiplier $\lambda_{h,\omega}$ coincides with the value of $\widetilde{u}_{h,\omega}$ 
at the centre of each face $E \in \mathcal{E}$. Note that 
the recovered pressure approximation $\widetilde{u}_{h,\omega}
\notin H^1_0(D)$ in general. It should not be confused with the standard
continuous piecewise linear pressure approximation. 
The approximation error of $\widetilde{u}_{h,\omega}$ is bounded as follows.

\begin{theorem}
\label{th:P1}
Let $u_\omega$ be the pressure solution in
\eqref{pde1.1}-\eqref{pde1.2} and let $\widetilde{u}_{h,\omega} \in \cW_{1;h}$ be
defined by \eqref{def:P1}. Then, under the assumptions of Theorem
\ref{prop:ChScTe} and for all \ $0<s<\min(t,\frac{\pi}{\theta_{\max}}) <
1$,
\begin{equation}
\label{est:P1}
\Vert u_\omega-\widetilde{u}_{h,\omega} \Vert_{L^2(D)} \;\lesssim\; C_u(\omega)\; h^{2s}
\end{equation}
where $C_u(\omega)$ is defined in \eqref{eq:dualbound}. 
Moreover,  $\Vert u -\widetilde{u}_{h} \Vert_{L^p(\Omega,  L^2(D))} \lesssim h^{2s}$\,, for all $p < r$.
\end{theorem}

\begin{proof}
We follow the proof of \cite[Theorem 2.2]{ArnoldBrezzi:1985} in the lowest order case
for a fixed sample $\omega \in \Omega$. We will not make this dependence explicit
in the proof though.   
Let $\mathcal{I}_h^\text{CR}: W^{1,1}(D) \to \cW_{1;h}$ be the canonical 
FE interpolation operator (as defined e.g. in
\cite[\S~1.4]{ErnGuermond:2004}). Then, for any $v \in W^{1,1}(D)$,
\begin{equation}
\label{est1}
\int_E (v - \mathcal{I}_h^\text{CR} v) \, \text{d}E = 0\,, \qquad \text{for all} \ \ E \in \mathcal{E}.
\end{equation}

We now estimate the approximation error locally on each element $T \in \mathcal{T}_h$\,.
The triangle inequality gives
\begin{equation}
\label{split-error}
\Vert u - \widetilde{u}_{h} \Vert_{L^2(T)}
\leq
\Vert u - \mathcal{I}_h^\text{CR} u \Vert_{L^2(T)} + \Vert \mathcal{I}_h^\text{CR} u - \widetilde{u}_{h}\Vert_{L^2(T)} \
.
\end{equation}
We bound each term on the right-hand side of the expression above separately.
First note that (cf. \cite[Thm ~1.103]{ErnGuermond:2004}), for all $v \in H^k(T)$ with $k=1,2$,
\[
\Vert v - \mathcal{I}_h^\text{CR} v \Vert_{L^2(T)} \lesssim h_T^{k} |v|_{H^{k}(T)}\,,
\]
and recall that the pressure $u \in H^{1+s}(D)$, with 
$0<s<\min(t,\frac{\theta_{\max}}{\pi}) < 1$ 
(cf.~Theorem~\ref{prop:ChScTe}). Therefore, since $H^1(T) \subset W^{1,1}(T)$, an operator 
interpolation argument between $H^1(T)$ and $H^2(T)$ (cf. \cite[Appendix B]{McLean}) allows 
us to conclude that the first term in \eqref{split-error} is bounded by
\begin{equation}
\label{project-error}
\Vert u - \mathcal{I}_h^\text{CR} u \Vert_{L^2(T)} = 
\Vert (I - \mathcal{I}_h^\text{CR}) u \Vert_{L^2(T)} \lesssim 
\ h_T^{1+s} \ \Vert u \Vert_{H^{1+s}(T)}\, .
\end{equation}
To estimate the second term in \eqref{split-error}, note that it follows from the definition of $P_h^\mathcal{E}$, as
well as from \eqref{def:P1} and \eqref{est1} that
\[
\int_E (\widetilde{u}_{h} - \mathcal{I}_h^\text{CR} u) \, \text{d}E = \int_E (\chi_T(\lambda_{h}) - u) \, \text{d}E =
\int_E (\lambda_{h} - P_h^\mathcal{E} u) \, \text{d}E\,, \qquad \text{for all} \ \ E \subset \partial T.
\]
Thus, using Lemma~\ref{lem:pwlin} with $w = \lambda_h-P_h^\cE u$ and $\chi_T = 
\widetilde{u}_{h} - \mathcal{I}_h^\text{CR} u$ and combining it with the estimate 
in Lemma~\ref{lem:M-1}, we obtain for each element $T \in \cT_h$,
$$
\Vert \widetilde{u}_{h} - \mathcal{I}_h^\text{CR} u \Vert_{L^2(T)} \lesssim
h_T^{1/2} \sum_{E \subset \partial T} \Vert \lambda_h - P_h^\cE u 
\Vert_{L^2(E)} \lesssim \frac{1}{a_{\min}(\omega)} h_T \Vert \vq -
\vq_h  \Vert_{L^2(T)} + \Vert u_h - P_h u  \Vert_{L^2(T)} \,.
$$
Using this together with \eqref{project-error} in \eqref{split-error}, squaring and summing over all $T \in \cT_h$, the
estimate \eqref{est:P1} follows from Theorem \ref{cor:mfe5} and Lemma \ref{th:mfe6}, since $s < 1$.
The bound on the moments follows due to Assumptions A1--A3 and H\"{o}lder's inequality.
\end{proof}

\begin{remark}
Note that for diffusion coefficients with trajectories in
$\mathcal{C}^{t}(\bar{D})$ with $t \le 1/2$ (e.g. for the
exponential covariance), for domains $D$ with
reentrant corners with $\theta_{\max} \approx 2\pi$, or for source terms $f$ that are only in
$L^2(D)$, the recovered pressure approximation
$\widetilde{u}_{h} \in \cW_{1;h}$ converges with the same rate as the piecewise
constant approximation $u_h \in \cW_h$ (see Theorem~\ref{th:super-p0}). 
\end{remark}

\begin{remark}
It is possible to prove the bound in \eqref{est:P1} without using the intermediate result in Lemma~\ref{lem:M-1}. 
(See \cite[\S~7.4]{BBF:2013} for details.) 
\end{remark}

\section{Application in the analysis of multilevel Monte Carlo methods}\label{sec:MLMC}

We apply the mixed FE error analysis carried out in Section~\ref{sec:error} to the complexity analysis of
\textit{multilevel Monte Carlo methods}.
Crucially, the FE error convergence rate determines the computational cost of multilevel Monte Carlo. 
It turns out that MLMC estimators are significantly more efficient than standard Monte Carlo 
for the groundwater flow problem.
We use MLMC to estimate the expected value $\mathbb{E}[Q]$ of certain functionals $Q:=\mathcal{G}(\vq,u)$ of the solution to the lognormal diffusion problem \eqref{pde1.1}-\eqref{pde1.2}. 
This approach is quite general since many important solution statistics, such as moments or failure probabilities, can be expressed in terms of expectations.

We fix a realisation $\omega \in \Omega$ and approximate the solution $(\vq_\omega,u_\omega)$ to \eqref{pde1.1}-\eqref{pde1.2} by the associated mixed FE approximation $(\vq_{h,\omega},u_{h,\omega})$ which satisfies \eqref{3.4}.
This leads us to approximate the functional $Q$ by $Q_h(\omega):=\mathcal{G}(\vq_{h,\omega},u_{h,\omega})$
evaluated on a FE mesh with mesh size $h$.
A common method to estimate the expected value of $Q_h(\omega)$ is the \textit{standard Monte Carlo estimator} for
$\mathbb{E}[Q_h]$, defined as
$$
\hat{Q}_{h,N}^{MC} := \frac{1}{N} \sum_{n=1}^N Q_h(\omega^{(n)}),
$$
where $Q_h(\omega^{(n)})$ is the functional corresponding to sample $\omega^{(n)}$.
Note that we compute $N$ statistically independent samples in total.

Let $X:\Omega \rightarrow \mathbb{R}$ denote a random variable.
In statistics it is common to quantify the accuracy of an estimator $\hat{X}$ to $\mathbb{E}[X]$ by the root mean square error (RMSE)
$$
e(\hat{X})^2:=\mathbb{E}[\hat{X}-\mathbb{E}[X]]^2 \ .
$$
The associated computational cost $C_\varepsilon(\hat{X})$ is characterised by the number of floating point operations required to achieve a RMSE $e(\hat{X})\leq \varepsilon$. 

The standard MC estimator is unbiased, $\mathbb{E}[\hat{Q}_{h,N}^{MC}]=\mathbb{E}[Q_h]$, and its variance is given by  $\mathbb{V}[\hat{Q}_{h,N}^{MC}] = N^{-1}\mathbb{V}[Q_h]$.
It is easy to see that these facts allow us to expand the mean square error (MSE) as
\begin{equation}
\label{RMSE}
e(\hat{Q}_{h,N}^{MC})^2 = N^{-1}\mathbb{V}[Q_h] + (\mathbb{E}[Q_h-Q])^2 \ .
\end{equation}
The first term above is the \textit{sample error} and is mainly controlled by the sample size $N$.
The second term, often referred to as \textit{bias}, is determined solely by the FE approximation error.
To ensure $e(\hat{Q}_{h,N}^{MC})\leq \varepsilon$ it is sufficient that both terms in \eqref{RMSE} are smaller than $\varepsilon^2/2$.
For the sample error this can be achieved by choosing $N=O(\varepsilon^{-2})$.
For the bias to be of order $|\mathbb{E}[Q_h-Q]|= O(\varepsilon)$, the FE mesh size has to be chosen sufficiently small.
In the context of our groundwater flow problem this is a tall order.
Since realisations of the diffusion coefficient are spatially rough and highly oscillatory we need a very fine FE mesh to obtain acceptable accuracies of the solution.
At the same time, the number of samples is in general quite large due to the slow convergence of the Monte Carlo estimator.

The MLMC estimator overcomes this difficulty by estimating the expected value not only on a single FE mesh with fixed mesh size but on a \textit{hierarchy} of increasingly finer FE meshes $\{\mathcal{T}_{h_\ell}\}_{\ell=0,\dots,L}$ with $\mathcal{T}_h:=\mathcal{T}_{h_L}$ the finest mesh, and $h_{\ell}/h_{\ell-1} \leq c < 1$.
Observe that by linearity of the expectation we may write
$$
\mathbb{E}[Q_h] = \mathbb{E}[Q_{h_0}] + \sum_{\ell=1}^L \mathbb{E}[Q_{h_\ell}-Q_{h_{\ell-1}}] \ .
$$
Let $Y_0:=Q_{h_0}$ and $Y_\ell:=Q_{h_\ell}-Q_{h_{\ell-1}}$, $\ell=1,\dots,L$.
The \textit{multilevel Monte Carlo estimator} for $\mathbb{E}[Q_h]$ is then defined as
$$
\hat{Q}_{h,N}^{ML} := \sum_{\ell=0}^L \hat{Y}_{\ell,N_{\ell}}^{MC}
=
\sum_{\ell=0}^L \frac{1}{N_\ell} \sum_{n=1}^{N_\ell} Y_\ell(\omega^{(n)}) \ .
$$
Since each correction $Y_\ell$, $\ell=0,\dots,L$ is estimated independently from the others, the RMSE of the MLMC estimator reads
$$
e(\hat{Q}_{h,N}^{ML})^2 = \sum_{\ell=0}^L N_{\ell}^{-1}\mathbb{V}[Y_\ell] + (\mathbb{E}[Q_h-Q])^2 \ .
$$
Notably, the bias associated with the MLMC estimator has not changed compared to the standard MC estimator. 
However, the sample error can now be controlled and distributed over the entire hierarchy of FE meshes. 
Importantly, the variance of the differences $Y_{\ell} \rightarrow 0$ as $h_\ell \rightarrow 0$ and hence the number of samples $N_{\ell}$ required on the finer meshes is very small. 
Essentially, the MLMC estimator allows us to shift a large part of the computational effort to coarse, inexpensive grids and requires only a small number of expensive fine grid simulations while maintaining the same accuracy (in terms of the RMSE).
A more detailed introduction to MLMC methods in the context of PDEs with random coefficients is presented in \cite{CGST:2011}.

\subsection{Multilevel Monte Carlo complexity}

Let $C_\ell$ denote the cost to obtain one sample of $Q_{h_\ell}$. 
The following result on the $\varepsilon$-cost of the MLMC estimator is taken from \cite[Theorem 1]{CGST:2011}.

\begin{theorem}
\label{MLMC-conv}
Let $\alpha, \beta, \gamma, c_{M_1}, c_{M_2}, c_{M_3}$ be positive constants such that $\alpha \geq \frac{1}{2} \min\{\beta,\gamma\}$ and 
\begin{itemize}
\item[\bf M1.]
$|\mathbb{E}[Q_{h_\ell}-Q]| \leq c_{M_1} h_\ell^\alpha$,
\item[\bf M2.]
$\mathbb{V}[Q_{h_\ell}-Q_{h_{\ell-1}}] \leq c_{M_2} h_\ell^\beta$,
\item[\bf M3.]
$C_\ell \leq c_{M_3} h_\ell^{-\gamma}$ \ .
\end{itemize}
Then, for any $\varepsilon < e^{-1}$, there exist a level $L$ and a sequence $\{N_\ell\}_{\ell=0}^L$, such that $e(\hat{Q}_{h,\{N_\ell\}}^{ML}) < \varepsilon$ and 
$$
C_\varepsilon(\hat{Q}_{h,\{N_\ell\}}^{ML}) 
\lesssim
\begin{cases}
\varepsilon^{-2}, &\text{if } \beta > \gamma,\\
\varepsilon^{-2}(\log \varepsilon)^2, &\text{if } \beta = \gamma,\\
\varepsilon^{-2-(\gamma-\beta)/\alpha},&\text{if } \beta < \gamma \ .
\end{cases}
$$
The hidden constant depends on $c_{M_1}, c_{M_2}, c_{M_3}$.
\end{theorem}

In Theorem~\ref{MLMC-conv}, M1-M2 describe assumptions on the spatial
discretisation error and must be verified for each output quantity in conjunction with a particular spatial discretisation.
We will prove these assumptions for the mixed FE discretisation of the lognormal diffusion problem and certain output
quantities in Section~\ref{sec:verify} ahead.
M3 is an assumption on the cost to obtain one sample of the output $Q_{h_\ell}$.
In our problem this is the cost of obtaining one sample of the lognormal diffusion coefficient plus the cost to solve the associated discretized PDE problem. 

Sampling the lognormal diffusion coefficient $a$ can be done by computing approximate Karhunen-Lo\`{e}ve eigenpairs of
the underlying Gaussian random field $\log(a)$ (see e.g., \cite{EEU:2007,SchwabTodor:2006}).
Then, a certain number $K_\ell$ of leading eigenpairs is retained in the expansion on level $\ell$ and the
truncated Karhunen-Lo\`eve expansion (KLE) of $\log(a)$ serves as approximation of the random field. 
The optimal choice of the truncation parameter $K_\ell$ that guarantees a negligible truncation error is
problem-dependent. Typically we expect $K_\ell\gtrsim h_\ell^{-m}$, where $m=1,2$ (cf. \cite[\S~4.1]{TSGU:2012}).
Exact samples of the underlying Gaussian field can be obtained by computing a factorisation of the covariance matrix associated with the quadrature nodes on the FE mesh.
A fast and efficient approach to do this is by circulant embedding (cf. \cite{Graham_etal:2011}) which has at most
log-linear complexity with respect to the number of quadrature points.
Thus, assuming that the cost of the PDE solver is of optimal order, that is, it scales linearly with respect to
the number of unknowns in the FE discretisation, then $\gamma \approx d$ (circulant embedding), or $\gamma\approx d+m$,
$m=1,2$ (truncated KLE).

The three upper bounds in Theorem~\ref{MLMC-conv} correspond to three scenarios.
Depending on the ratio of $\beta$ and $\gamma$ in Assumptions M2 and M3, the major part of the computational cost could be on the coarsest level ($\beta >
\gamma$), spread evenly across all levels ($\beta = \gamma$) or on the finest level ($\beta <
\gamma$).

In the context of realistic groundwater flow applications in two and three space dimensions the costs to obtain one
sample grow rapidly with decreasing spatial resolution and thus we will almost always be in the last regime $\beta <
\gamma$. 
If $\beta = 2 \alpha$ (as is often the case, see Section~\ref{sec:verify} ahead), then, the total cost
of the MLMC estimator is of order $\varepsilon^{-\gamma/\alpha}$ which is asymptotically the same as the cost to compute only
\textit{one} sample on the finest mesh to accuracy $\varepsilon$. 
The gains we can expect by using the MLMC estimator in place of the standard Monte Carlo estimator are thus
significant and when $\beta = 2 \alpha$ the MLMC estimator is asymptotically optimal.

\subsection{Verifying the mixed FE error convergence rates}\label{sec:verify}

The performance analysis of the MLMC method as stated in Theorem~\ref{MLMC-conv} relies on bounds for $|\mathbb{E}[Q_{h_\ell}-Q]|$ and $\mathbb{V}[Q_{h_\ell}-Q_{h_{\ell-1}}]$ in terms of the characteristic mesh size $h_\ell$ on level $\ell$.
These bounds can be proved by bounding the spatial discretisation error.
See \cite{CST:2011,TSGU:2012} for an analysis of MLMC for \eqref{pde1.1}-\eqref{pde1.2} in the framework of
standard FEs. 

By using the appropriate error estimates derived in Section~\ref{sec:error} we can now easily deduce the convergence rates $\alpha$ and $\beta$ in Theorem~\ref{MLMC-conv} for various quantities of interest $Q$ and
the associated FE approximation $Q_h$. 
In Proposition \ref{prop:MLMC1} below, we use the following simple Lemma.

\begin{lemma}
\label{lem:simple}
Let $\mathcal{B}$ denote a Banach space with norm $\Vert \cdot \Vert_\mathcal{B}$.
Let $X,Y  \in L^2(\Omega,\mathcal{B})$ be $\mathcal{B}$-valued random variables. 
Then we have
\begin{eqnarray}
\left|\mathbb{E}[\Vert X \Vert_{\mathcal{B}}-\Vert Y \Vert_{\mathcal{B}}]\right|
&\leq& \Vert X-Y \Vert_{L^1(\Omega,\mathcal{B})},\\
\mathbb{V}[\Vert X \Vert_{\mathcal{B}}-\Vert Y \Vert_{\mathcal{B}}] &\leq& \Vert X-Y \Vert^2_{L^2(\Omega,\mathcal{B})}\ .
\end{eqnarray}
\end{lemma}

\begin{proof}
It is easy to see that
$$
\left|\mathbb{E}[\Vert X \Vert_{\mathcal{B}}-\Vert Y \Vert_{\mathcal{B}}]\right|
\leq
\mathbb{E}[|\Vert X \Vert_{\mathcal{B}}-\Vert Y \Vert_{\mathcal{B}}|]
\leq
\mathbb{E}[\Vert X - Y \Vert_{\mathcal{B}}]
=
\Vert X-Y\Vert_{L^1(\Omega,\mathcal{B})} \ ,
$$
where, in the second step, we have used the reverse triangle inequality.
To bound the variance of the difference we use $\mathbb{V}[X] = \mathbb{E}[X^2] - (\mathbb{E}[X])^2 \leq \mathbb{E}[X^2]$, and, again, the reverse triangle inequality:
$$
\mathbb{V}[\Vert X \Vert_{\mathcal{B}}-\Vert Y \Vert_{\mathcal{B}}]
\leq 
\mathbb{E}[(\Vert X \Vert_{\mathcal{B}}-\Vert Y \Vert_{\mathcal{B}})^2]
\lesssim
\mathbb{E}[\Vert X - Y \Vert_{\mathcal{B}}^2]
=
\Vert X - Y \Vert_{L^2(\Omega,\mathcal{B})}^2 \ .\vspace{-3ex}
$$
\end{proof}

\begin{proposition}\label{prop:MLMC1}
Let $D \subset \mathbb{R}^2$ be a polygon with largest interior angle $\theta_{\max} \in(0,2\pi)$. 
Let Assumptions A1-A3 hold for some $0<t< 1$ and $r>2$. 
Define $t^\star:=\min(t,\frac{\pi}{\theta_{\max}})$.
Then, assumptions M1-M2 hold with $\alpha$ and $\beta$ as follows:
\begin{center}
\begin{tabular}{llll}\hline\noalign{\smallskip}
$Q$ & $Q_h$ && \textnormal{Reference}\\\noalign{\smallskip}\hline\noalign{\smallskip}
$\Vert\vq\Vert_{L^2(D)}$ & $\Vert\vq_h\Vert_{L^2(D)}$&$\alpha<t^\star$,
$\beta<2 t^\star$&\textnormal{Corollary~\ref{cor:velocity}} \\
$\Vert\vq\Vert_{H(\rdiv)}$ & $\Vert\vq_h\Vert_{H(\rdiv)}$&$\alpha<t^\star$,
$\beta<2t^\star$&\textnormal{Corollary~\ref{cor:velocity}} \\
$\Vert u\Vert_{L^2(D)}$ & $\Vert
u_h\Vert_{L^2(D)}$&$\alpha<\min(1,2t^\star)$, $\beta<\min(2,4t^\star)$&\textnormal{Theorem~\ref{th:super-p0}}\\
$\Vert u\Vert_{L^2(D)} $& $\Vert \widetilde{u}_{h}\Vert_{L^2(D)}$& $\alpha<2t^\star$,
$\beta<4t^\star$&\textnormal{Theorem~\ref{th:P1}}\\
$\mathcal{M}(\vq)$ & $\mathcal{M}(\vq_h)$ &
$\alpha<2t^\star$, $\beta<4t^\star$&\textnormal{Theorem~\ref{th:mfe7}}\\\noalign{\smallskip}\hline
\end{tabular}
\end{center}
Here, $\mathcal{M}\in \cV^\star$ is a linear functional of the Darcy velocity $\vq$ in
\eqref{3.3}
that satisfies Assumption A4.
\end{proposition}

\begin{proof}
The first four cases follow by combining Lemma~\ref{lem:simple} 
with the corresponding error estimates in Section~\ref{sec:error},
choosing $X$, $Y$ and $\mathcal{B}$ accordingly.

For linear velocity functionals $\mathcal{M}$ we obtain
\[
|\mathbb{E}[\mathcal{M}(\vq-\vq_h)]|
\leq
\mathbb{E}[|\mathcal{M}(\vq-\vq_h)|]
=
\Vert \mathcal{M}(\vq-\vq_h)\Vert_{L^1(\Omega)} 
\] 
and
\[
\mathbb{V}[\mathcal{M}(\vq-\vq_h)]
\leq
\mathbb{E}[\mathcal{M}^2(\vq-\vq_h)]
=
\Vert \mathcal{M}(\vq-\vq_h)\Vert_{L^2(\Omega)}^2 \ . 
\] 
The assertion follows by combining these bounds with the error estimates in Theorem~\ref{th:mfe7}.
\end{proof}
\begin{remark}
As stated before Theorem~\ref{prop:ChScTe}, similar results can be proved for $D\subset \mathbb{R}^3$. 
\end{remark}


\section{Numerical experiments}\label{sec:numerics}

As a representative example, we consider the 2D ``flow cell'' problem in mixed formulation
\begin{align}\label{modnum}
a^{-1}(\omega,\vx) \vq(\omega,\vx) + \nabla u(\omega,\vx) &= 0,&\\ 
\nabla \cdot \vq(\omega,\vx) &= 0, &\quad \mathrm{in} \quad D=(0,1) \times (0,1) \nonumber.
\end{align}
The horizontal boundaries are no-flow boundaries, that is, $\vec{\nu} \cdot \vq = 0$. We have $u\equiv 1$ along the
western (inflow) and $u\equiv 0$ along the eastern (outflow) boundary.

The diffusion coefficient $a(\omega, x)$ is a lognormal random field; $\log(a)$ is a mean-zero Gaussian random field with variance $\sigma^2\equiv 1$ and a specific covariance function $\rho$.
In our examples we will use the exponential covariance
\begin{equation}\label{exp}
\rho(r) = \rho_\mathrm{exp}(r):=\sigma^2 \ \exp(-r/\lambda),
\end{equation}
where $r=\|\vec{x}-\vec{y}\|_2$ is the Euclidean distance of $\vec{x},\vec{y} \in \mathbb{R}^d$, and $\lambda>0$ denotes
the correlation length.
We will also consider the Mat\'ern covariance function 
\begin{equation}\label{matern}
\rho(r) = \rho_\nu(r):=\sigma^2 \ \frac{2^{1-\nu}}{\Gamma(\nu)}\left(\frac{r}{\widetilde{\lambda}}\right)^\nu
K_\nu\left(\frac{r}{\widetilde{\lambda}}\right),
\end{equation}
where $K_\nu$ is the modified Bessel function of second kind and order $\nu$, and
$\widetilde{\lambda}=\lambda/(2\sqrt{\nu})$ denotes the scaled correlation length.
Using the asymptotics of $K_{0.5}$ it is possible to show that $\rho_{0.5}$ is 
actually the exponential covariance $\rho_\mathrm{exp}$ with correlation length 
$\widetilde{\lambda}$ instead of $\lambda$.

The physical discretisation of the weak formulation associated with \eqref{modnum} is done with lowest order
Raviart-Thomas mixed finite elements for the Darcy velocity $\vq$ and piecewise constant elements for the pressure $u$
(see Section~\ref{sec:error}) on a uniform mesh of $n \times n$ squares, each divided into two triangles.
We solve the resulting saddle point problems using the efficient divergence-free reduction technique for
Raviart-Thomas finite elements introduced in \cite{CGSS:2000, Scheichl:2000}.
The associated symmetric positive-definite linear system is solved with the sparse direct solver implemented
in Matlab. The theoretical cost of this is $O(h^{-3})$, but in practice it is often faster.
In addition, we compute the recovered pressure approximation $\widetilde{u}_{h}$ from Section \ref{sec:err:P1} and the
effective permeability
\begin{equation}\label{keff}
k_{\textnormal{eff}}(\omega)=\frac{\int_D q_1(\omega,\vx)\, \dx}{- \int_D \frac{\partial u}{\partial
x_1}(\omega,\vx)\dx},
\end{equation}
which simplifies to $\int_D q_1(\omega,\vx)\dx$ for the flow cell problem (see \cite{Graham_etal:2011}).

For the exponential covariance we generate samples of $\log(a)$ at the vertices of the FE mesh using the circulant embedding technique (see, e.g. \cite[section 5]{Graham_etal:2011}).
For the Mat\'ern covariance we use a truncated Karhunen-Lo\`eve expansion (KLE) where
only a certain number of the leading eigenpairs is retained. 
The eigenpairs are approximated by a spectral collocation method \cite{Atkinson:1997}.

\begin{remark} 
The analysis in Sections \ref{sec:reg} and \ref{sec:error} has been performed for pure 
Dirichlet boundary conditions.
The flow cell test problem features mixed Dirichlet/Neumann boundary conditions. 
This case is handled in \cite[Remark 5.4 (b)]{TSGU:2012} for the primal formulation of 
\eqref{modnum} and this analysis carries over to mixed formulations. 
Importantly, in the flow cell problem, no additional singularities are introduced at 
the points where the Dirichlet and Neumann boundary segments meet.  
To see this, we reflect the problem and the solution across the Neumann boundary and 
apply the regularity theory to the union of the original and reflected domain. 
Crucially, all angles of the new domain are less than $\pi/2$ and thus the full 
regularity of the pressure is maintained.
\end{remark}

\begin{remark}
In hybridised mixed methods it is standard to compute the Lagrange multipliers first
and subsequently recover the Darcy velocity and pressure by local post-processing.
The piecewise linear pressure recovery is then also a cheap, local procedure.
Similarly, in non-hybridised methods, the Lagrange multipliers and thus the piecewise 
linear pressure approximation $\widetilde{u}_{h,\omega}$ can be recovered cheaply by 
substituting the Darcy 
velocity $\vq_{h,\omega} \in \cV_h \subset RT_{-1}(\mathcal{T}_h)$ and the pressure 
$u_{h,\omega} \in \cW_h$ into the first equation in \eqref{P1-system} and solving for 
$\lambda_{h,\omega}$. Note that for the usual choices of basis functions in 
$RT_{-1}(\mathcal{T}_h)$, $\cW_h$ and $M_{-1}(\mathcal{E}_I)$, this recovery is completely 
local. 
\end{remark}

\subsection{Mixed FE error convergence rates}

We investigate the convergence of FE approximations to \eqref{modnum} with respect to the characteristic mesh size
$h_{\ell}=n_\ell^{-1}$, on a sequence of uniform meshes with $n_\ell=n_0*2^{\ell}$, $\ell=0,\dots,L$. 
We use a Monte Carlo method with $N=2000$ samples for the convergence tests. 
The reference solutions $\vq_{\star}$ and $u_{\star}$ are computed on a grid with $h_\star = 1/256$.

First, we consider the exponential covariance function $\rho=\rho_\mathrm{exp}$ in \eqref{exp}.
In this case the trajectories of $\log(a)$ (and thus $a$) belong to $\mathcal{C}^{t}(\bar{D})$ almost surely for
all $t < 1/2$ and Assumption A2 is satisfied for all $0<t<1/2$.
Note that $\Div \,\vq \equiv 0$ and thus $\|\vq\|_{H(\Div,D)} = \|\vq\|_{L^2(D)}$.
Consequently, our theory in Section~\ref{sec:error} tells us to expect
$\Vert\vq_{\star}-\vq_h\Vert_{L^2(\Omega,H(\rdiv,D))} = O(h^{1/2-\delta})$ 
and $\Vert u_{\star}-u_{h}\Vert_{L^2(\Omega,L^2(D))} = O(h)$, for all $\delta>0$.
We expect essentially the same convergence $\Vert u_{\star}-\widetilde{u}_{h}\Vert_{L^2(\Omega,L^2(D))} =
O(h^{1-\delta})$ for the recovered pressure approximation.
For the effective permeability we also expect $|k_{\textnormal{eff};\star}-k_{\textnormal{eff};h}|_{L^2(\Omega)} =
O(h^{1-\delta})$. 
The results in Figure~\ref{fig:test_conv_log_nu05_lam1} confirm our theory. 
We observe linear convergence for the standard and the recovered pressure approximation as well as the effective
permeability. 
The velocity approximation is of order $O(h^{1/2})$.

\begin{figure}[t]
\centering
\includegraphics[width=0.48\textwidth]{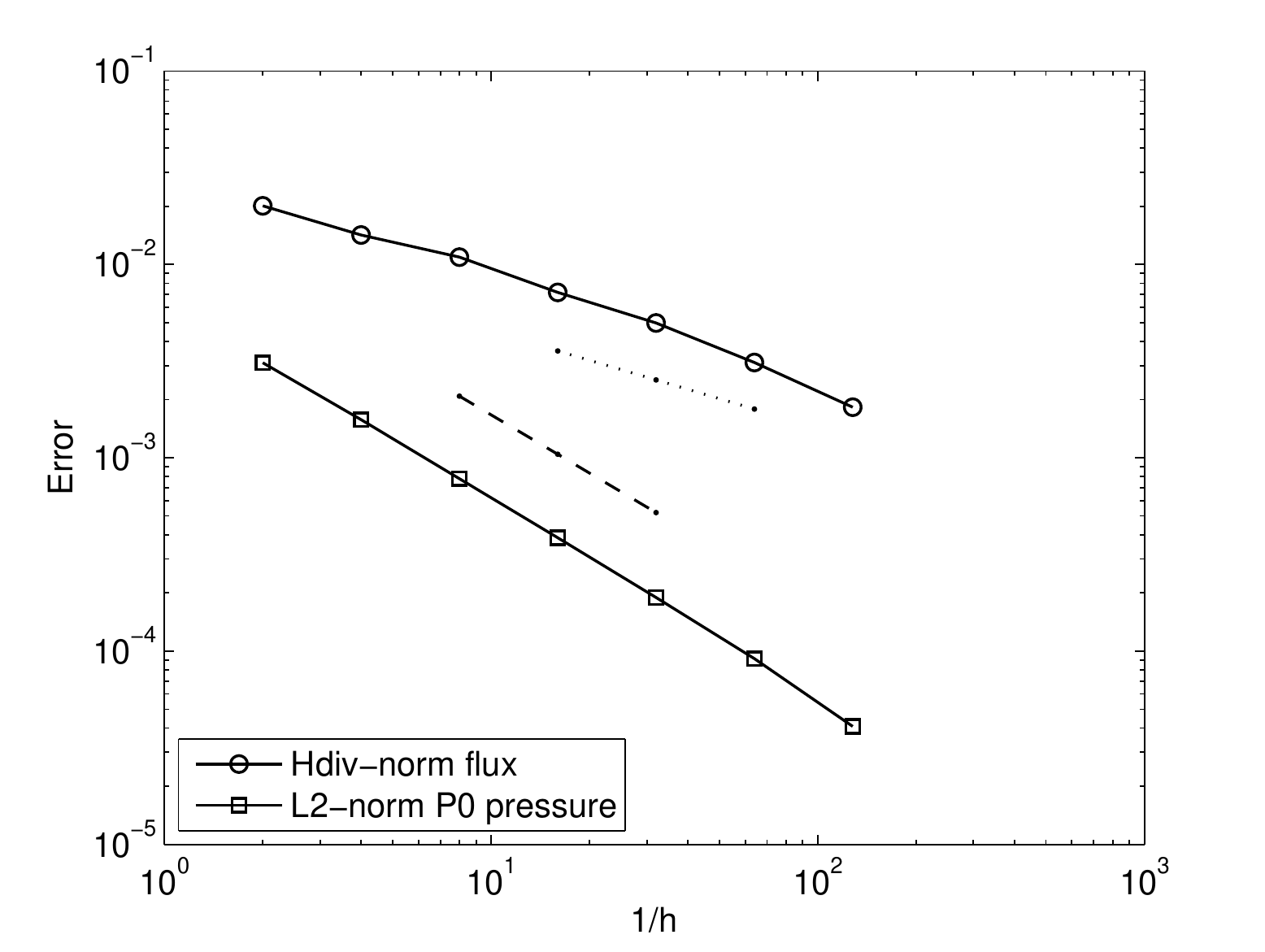}
\includegraphics[width=0.48\textwidth]{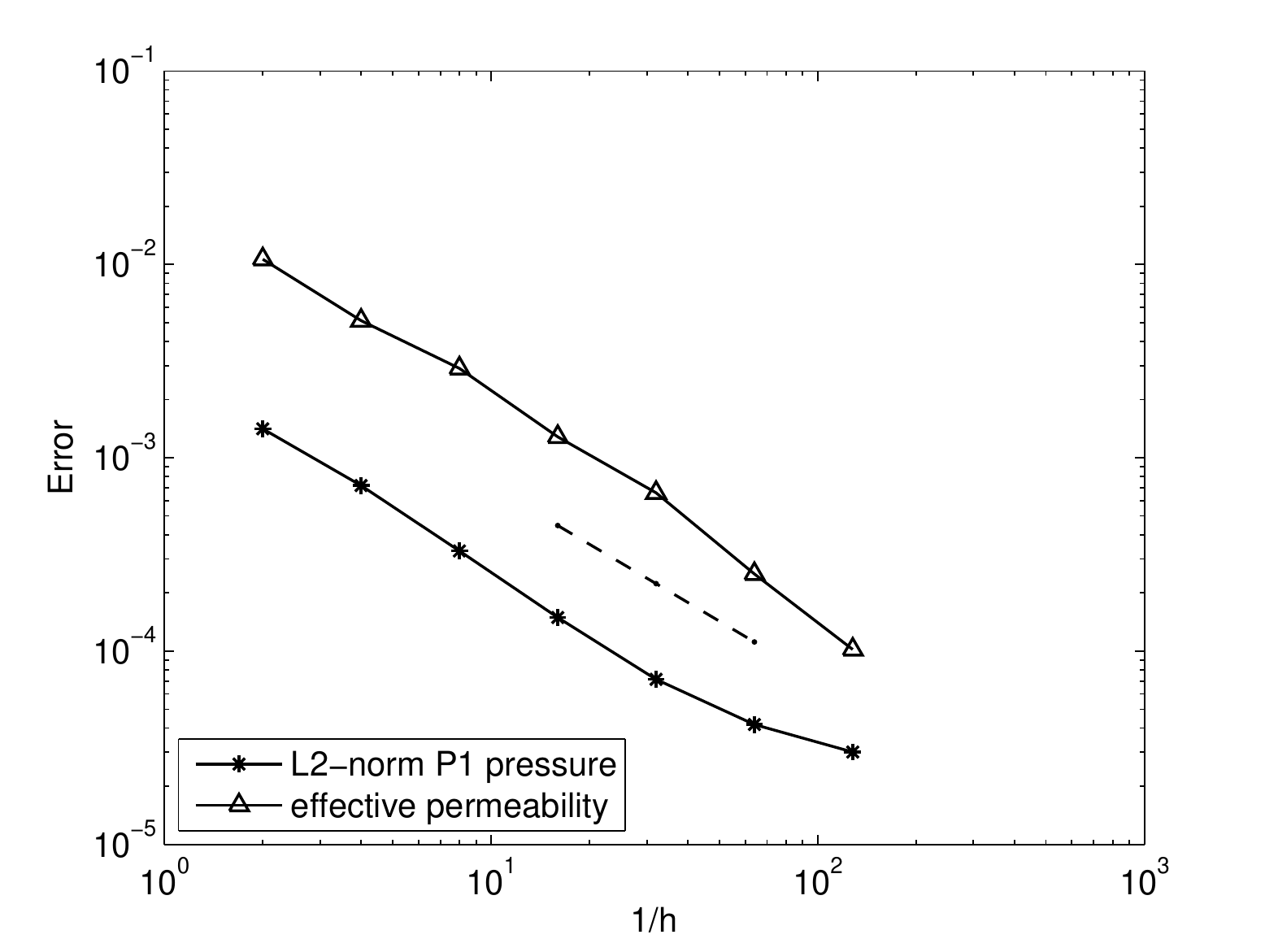}
\caption{Mixed FE approximation errors with respect to the mesh size 
$h_{\ell}$ for $\rho=\rho_\mathrm{exp}$ and $\lambda=1$. The dotted line has slope $-1/2$; the dashed line has slope
$-1$.}
\label{fig:test_conv_log_nu05_lam1}
\end{figure}

Next we consider $\log(a)$ with Mat\'ern covariance \eqref{matern} with $\nu=2$ and $\lambda=0.5$.
In this case it can be shown that the trajectories of $\log(a)$ belong to $\mathcal{C}^{1}(\bar{D})$ almost surely;
hence Assumption A2 is satisfied for $t=1$.
We retain the leading 13 eigenpairs of the KLE which captures more than 95\% of the variability of $\log(a)$.
Due to the theory in Section~\ref{sec:error} we expect linear convergence for both the velocity and standard pressure
approximation and quadratic convergence for the recovered pressure approximation and the
effective permeability.
The results in Figure~\ref{fig:test_conv_log_nu2_lam05} confirm this.
Note that these convergence rates are optimal for lowest-order Raviart Thomas mixed finite elements and can be achieved for all Mat\'ern covariances with parameter $\nu>1$.

\begin{figure}[t]
\centering
\includegraphics[width=0.48\textwidth]{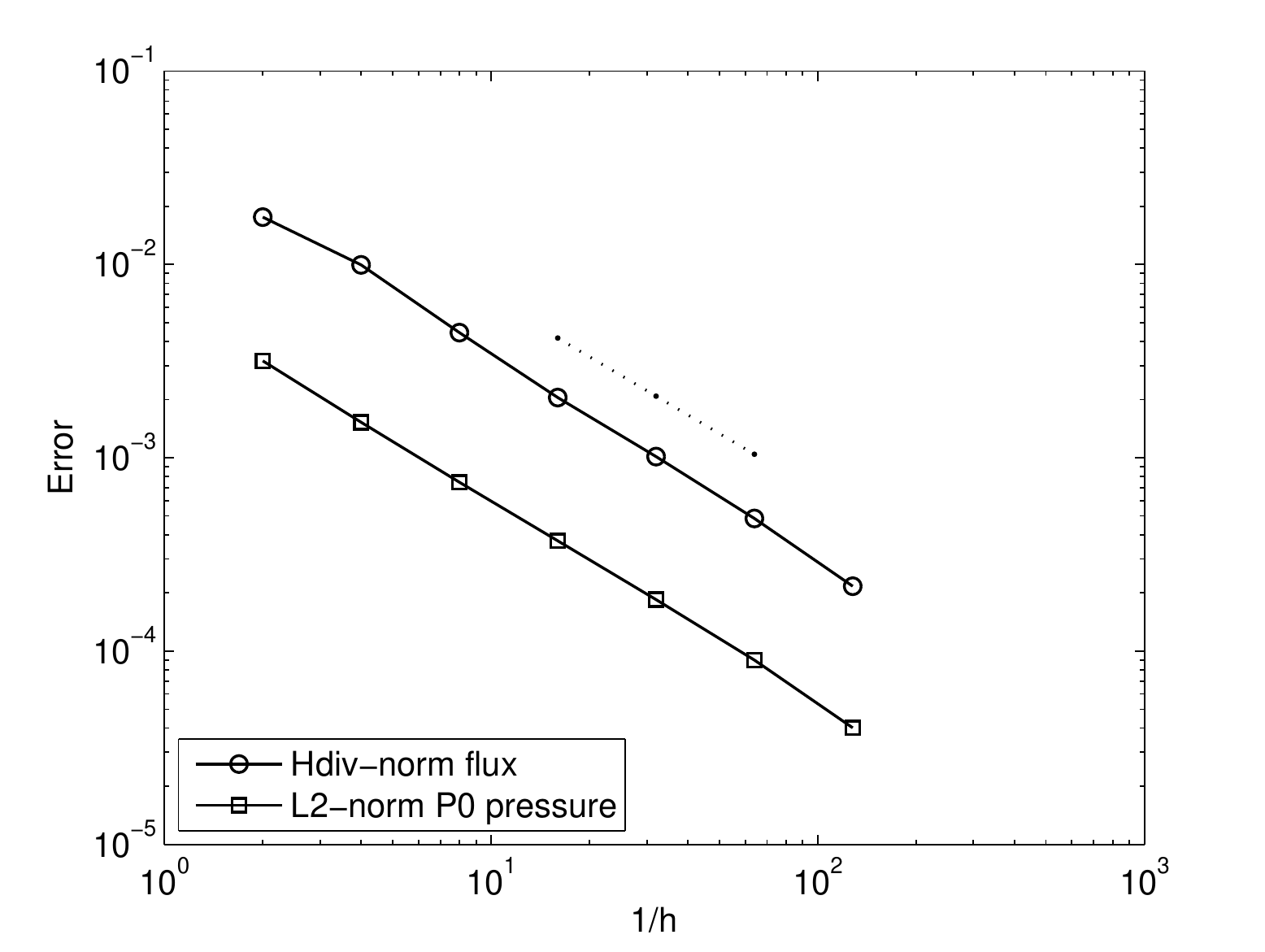}
\includegraphics[width=0.48\textwidth]{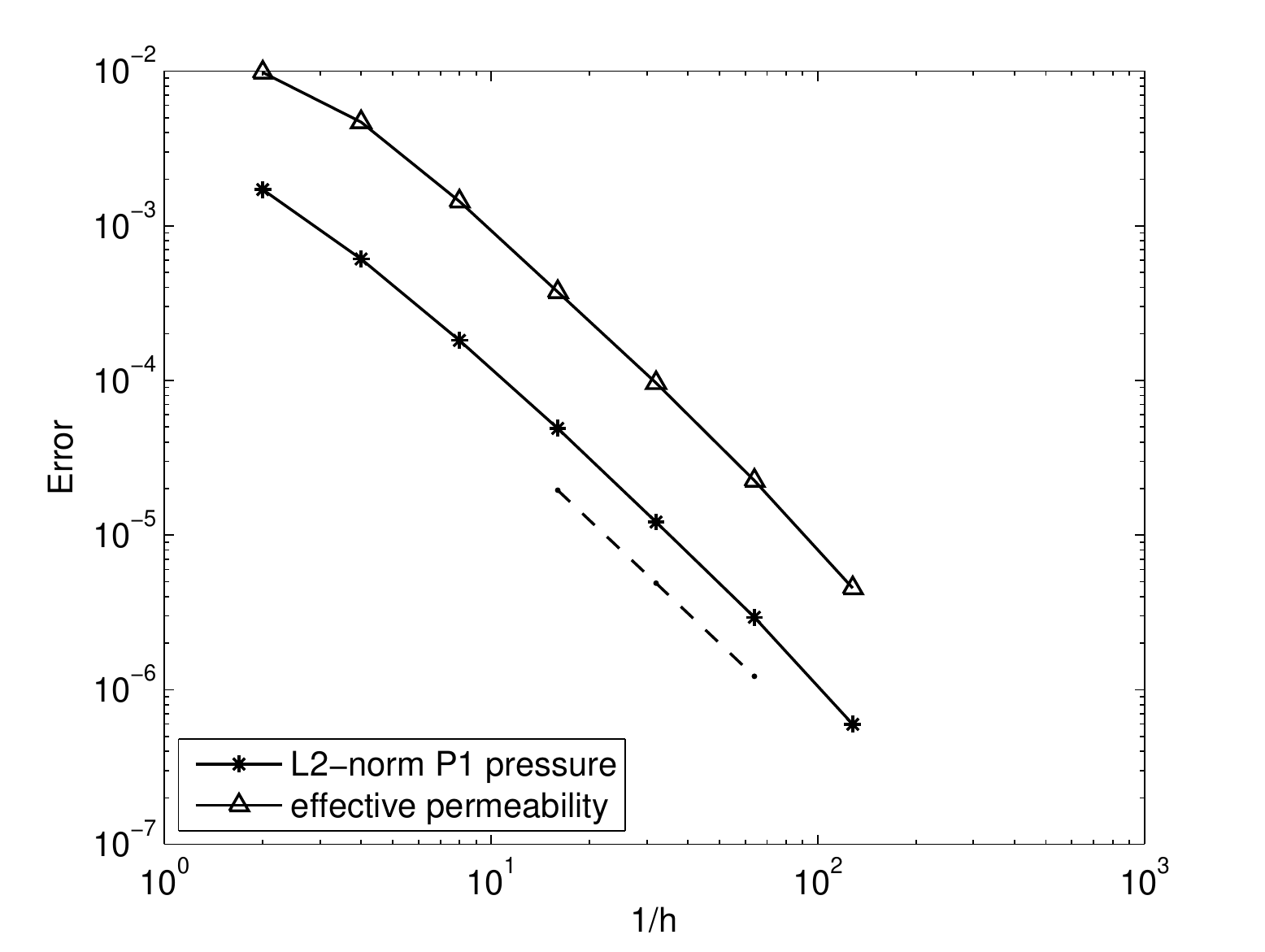}
\caption{Mixed FE approximation errors with respect to the mesh size 
$h_{\ell}$ for $\rho=\rho_2$ and $\lambda=0.5$. The dotted line has slope $-1$; the dashed line has slope $-2$.}
\label{fig:test_conv_log_nu2_lam05}
\end{figure}

\subsection{Multilevel Monte Carlo simulation}

We choose $\rho=\rho_\mathrm{exp}$ and $\lambda=0.1$ in our test problem.
This time, we investigate the theoretical assumptions (M1) and (M2) in the MLMC complexity 
theorem for various quantities of interest derived from solutions to \eqref{modnum} 
using $N=5000$ samples for the tests. 
The reference solution $(\vq_{\star},u_{\star})$ along with 
 $k_{\textnormal{keff},\star}$ are again computed on a grid with $h_\star=1/256$. 
The quantities of interest are the $H(\rdiv)$-norm of the Darcy velocity and the 
effective permeability $k_\text{eff}$ in \eqref{keff}. 

Since the trajectories of $a$ belong to $\mathcal{C}^{t}(\bar{D})$ almost surely for all $t < 1/2$, our theory in
Section~\ref{sec:MLMC} tells us to expect $|\mathbb{E}[\|\vq_h\|_{H(\rdiv)}-\|\vq_{\star}\|_{H(\rdiv)}]= O(h^{1/2})$ and
$\mathbb{V}[\|\vq_h\|_{H(\rdiv)}-\|\vq_{2h}\|_{H(\rdiv)}] = O(h)$. 
The results in Figure \ref{fig:exp_var_Hdiv} show that this result is not sharp. 
We observe twice the expected convergence rate.
This is because $\Vert\vq\Vert_{H(\rdiv)} = (\int_D |\vq|^2)^{1/2}$ here, which is a simple 
nonlinear functional of $\vq$ that can be analysed by the same techniques as presented 
in \cite{TSGU:2012} for standard piecewise linear FEs.
For the effective permeability we expect $|\mathbb{E}[k_{\textnormal{eff};h}-k_{\textnormal{eff};\star}]| = O(h)$ and $\mathbb{V}[k_{\textnormal{eff};h}-k_{\textnormal{eff};2h}] = O(h^2)$. 
This time, the results in Figure \ref{fig:exp_var_keff} confirm this theory.

\begin{figure}[th]
\centering
\includegraphics[width=0.48\textwidth]{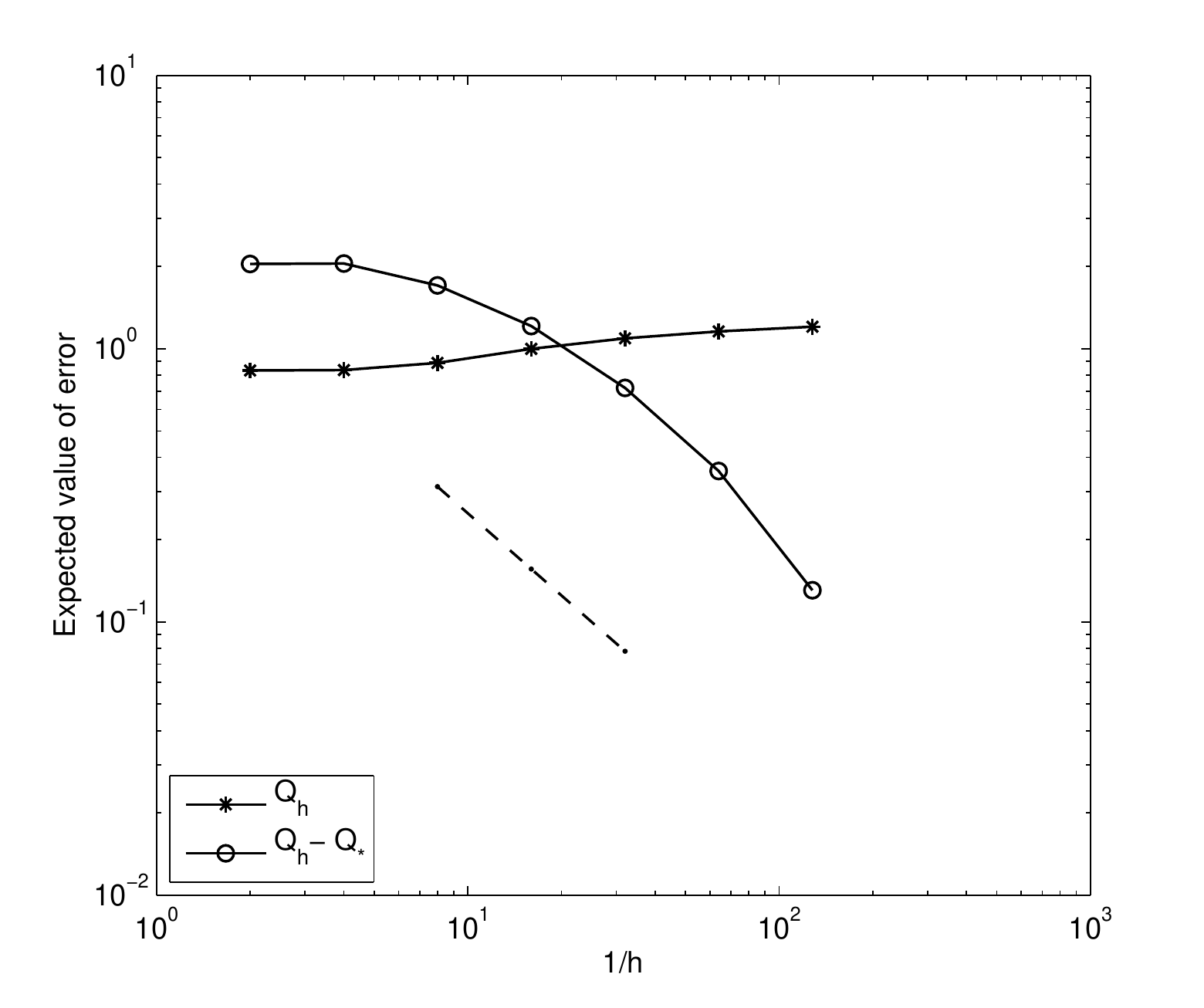}
\includegraphics[width=0.48\textwidth]{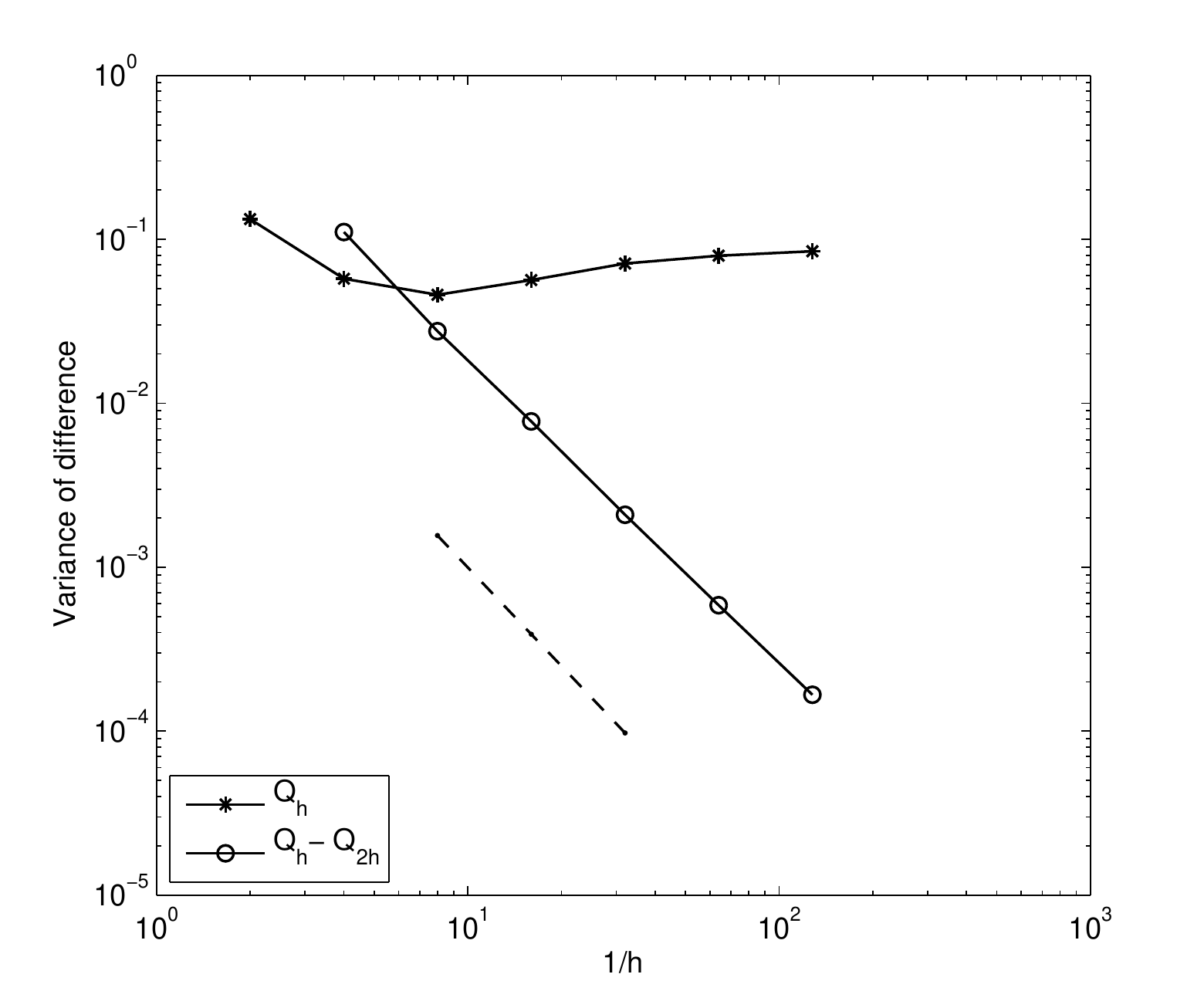}
\caption{Plot of $|\mathbb{E}[\|\vq_h\|_{H(\rdiv)}]|$ and
$|\mathbb{E}[\|\vq_h\|_{H(\rdiv)}-\|\vq_{\star}\|_{H(\rdiv)}]|$ (left), as well as $\mathbb{V}[\|\vq_h\|_{H(\rdiv)}]$
and
$\mathbb{V}[\|\vq_h\|_{H(\rdiv)}-\|\vq_{2h}\|_{H(\rdiv)}]$ (right) with respect to the mesh size $h$ for
$\rho=\rho_\mathrm{exp}$ and $\lambda=0.1$. The
dashed line has slope $-1$ (left) resp. $-2$ (right). }
\label{fig:exp_var_Hdiv}
\vspace{1ex}
\centering
\includegraphics[width=0.48\textwidth]{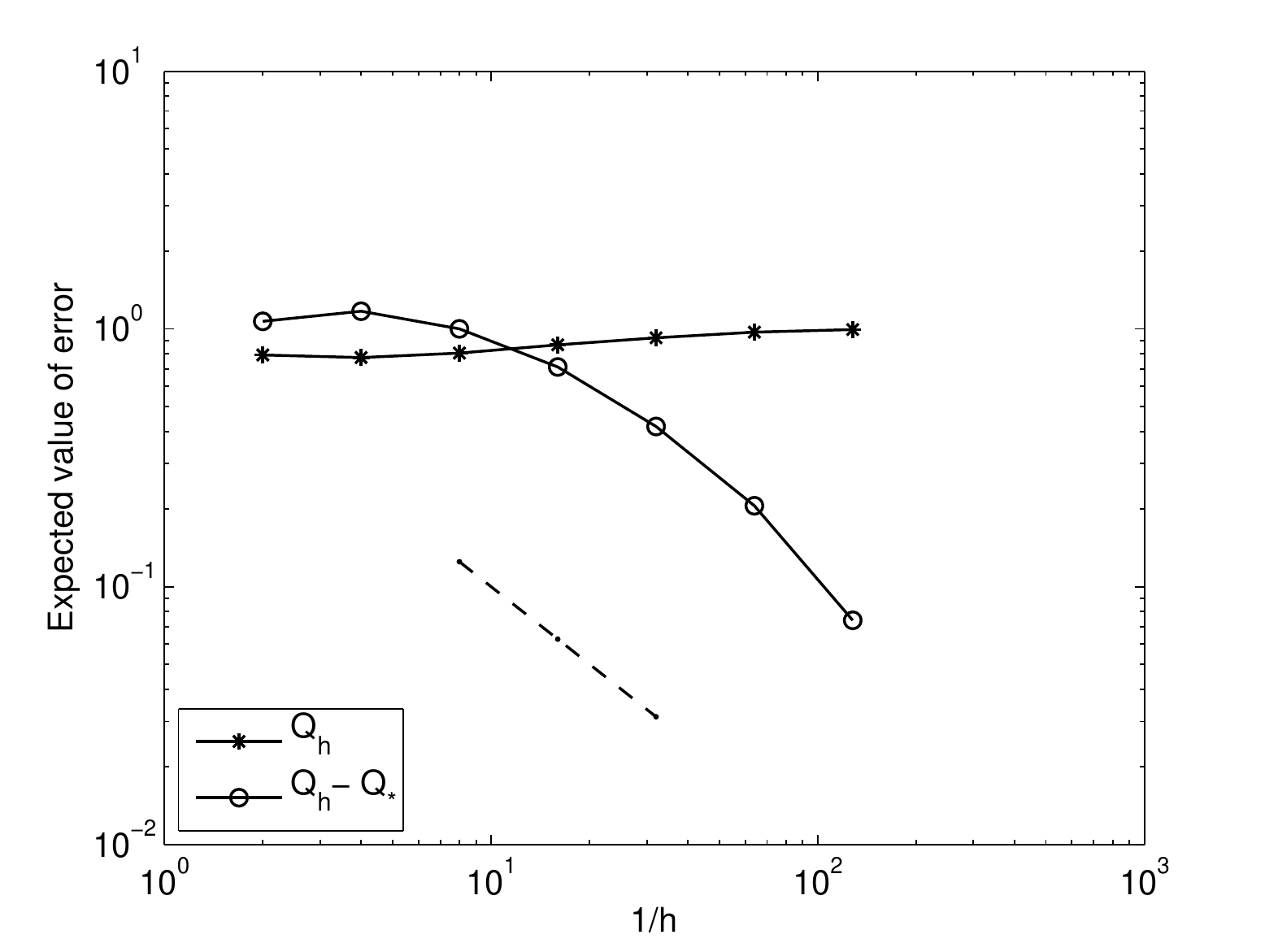}
\includegraphics[width=0.48\textwidth]{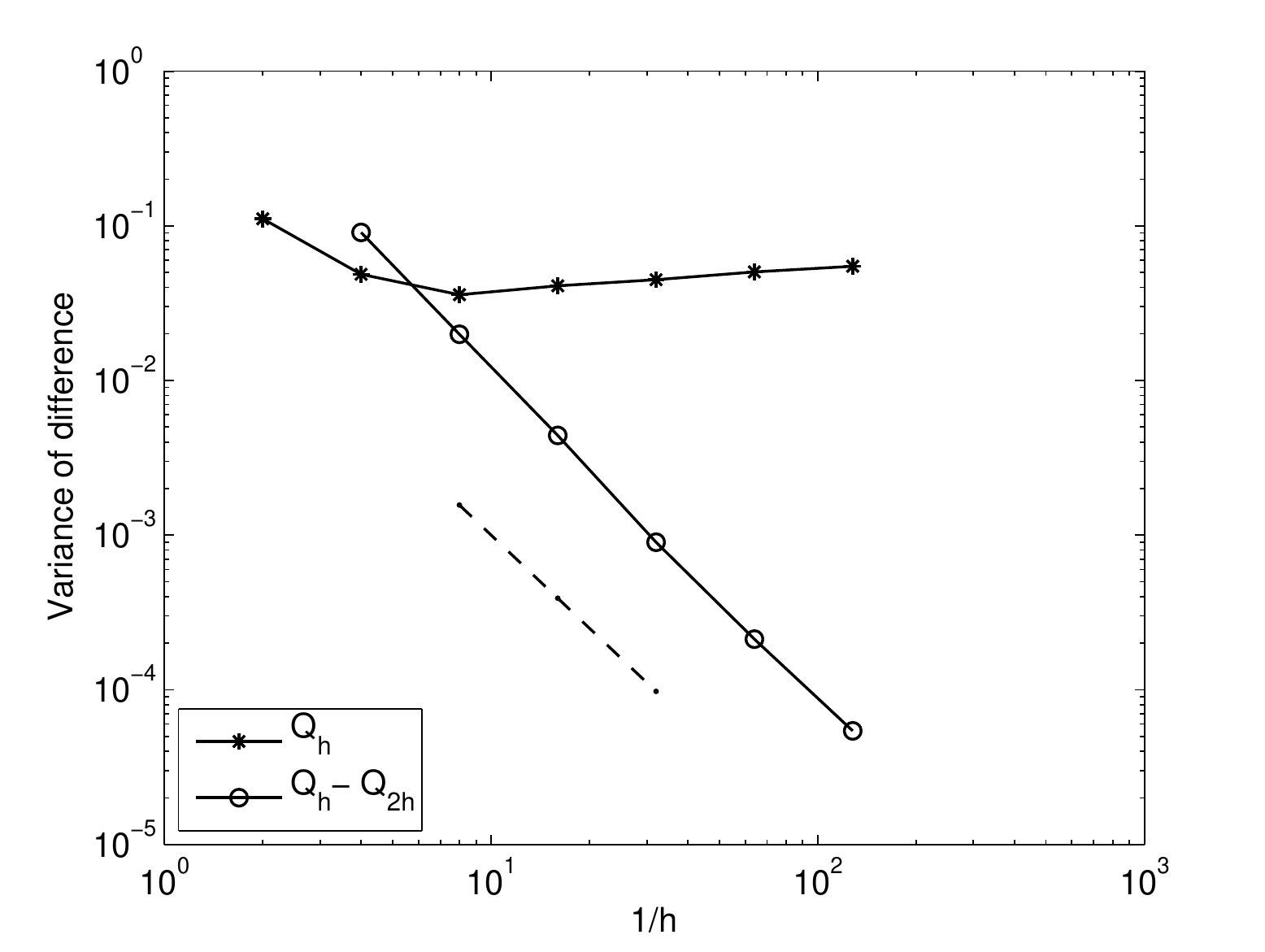}
\caption{Plot of $|\mathbb{E}[k_{\textnormal{eff};h}]|$ and
$|\mathbb{E}[k_{\textnormal{eff};h}-k_{\textnormal{eff};\star}]|$ (left), as well as
$\mathbb{V}[k_{\textnormal{eff};h}]$ and
$\mathbb{V}[k_{\textnormal{eff};h}-k_{\textnormal{eff};2h}]$ (right) with respect to the mesh size $h$
for $\rho=\rho_\mathrm{exp}$ and $\lambda=0.1$.
The dashed line has slope $-1$ (left) resp. $-2$ (right).}
\label{fig:exp_var_keff}
\end{figure}

Finally, in Figure \ref{fig:cpu:MLMC} we compare the computational costs of the standard and multilevel MC estimator, respectively, for the estimation of the expected value of the effective permeability \eqref{keff} where the sample error is fixed at $10^{-4}$ (i.e. $\varepsilon^2/2 = 10^{-4}$ and so
$\varepsilon = \sqrt{2} 10^{-2}$. 
The CPU timings (in seconds) shown in Figure~\ref{fig:cpu:MLMC} are calculated using Matlab 7.8 (with a single
computational thread) on an 8 processor Linux machine with 14.6 GByte of RAM.
As usual, the optimal number of samples $N_\ell$ in the MLMC estimator on each level
 is computed from the formula
\begin{equation}\label{optimalNl}
N_\ell := \left\lceil 2 \varepsilon^{-2}  \sqrt{\widehat{s}^2_\ell h_\ell} \sum_{\ell'=0}^L \sqrt{\widehat{s}^2_{\ell'}/h_{\ell'}} \right\rceil = 10^{-4} \left\lceil 
\sum_{\ell'=0}^L \sqrt{2^{\ell'-\ell} \, \widehat{s}^2_{\ell'}\widehat{s}^2_\ell} \right\rceil\,,
\end{equation}
where $\widehat{s}^2_\ell$ is the sample variance on level $\ell$, i.e. an estimate of $\mathbb{V}[Y_\ell]$ with an initial set of $\widetilde{N}_\ell < N_\ell$ samples (cf.~\cite[Section 5]{Giles:2008}). 

The efficiency of MLMC as compared to standard MC is clearly demonstrated. 
On a mesh with $h=1/256$ the 4-level MLMC estimator takes only 40 seconds, but the single level MC estimator takes 27 minutes.
We observe that in Figure~\ref{fig:exp_var_keff} the graphs of
$\mathbb{V}[k_{\textnormal{eff};h}-k_{\textnormal{eff};2h}]$ and $\mathbb{V}[k_{\textnormal{eff};h}]$ intersect at
$h \approx \lambda = 0.1$. This means that the cost of the MLMC estimator on any coarser mesh will actually be bigger
than the cost of the standard MC estimator on the same mesh. 
The coarsest mesh used in Figure~\ref{fig:cpu:MLMC} contains $32 \times 32$ elements.

This phenomenon was already observed in \cite{CGST:2011} and can be overcome by using smoother approximations of the permeability on coarser finite element meshes as suggested in \cite[Section 4]{TSGU:2012}, where so called
``level-dependent'' MLMC estimators are studied.
In \cite{TSGU:2012} the log-permeability is approximated by a truncated KL expansion where a decreasing number of modes
is included on the coarser meshes. 
A similar strategy has been suggested in the context of the related Brinkman problem in \cite{Gittelson_etal:2011} under the assumption of a certain decay rate of the finite element error with respect to the number of KL modes.
We do not study level-dependent MLMC estimators here since the focus of this work is on the mixed finite element error
estimation. 

\begin{figure}[t]
\centering
\includegraphics[width=0.48\textwidth]{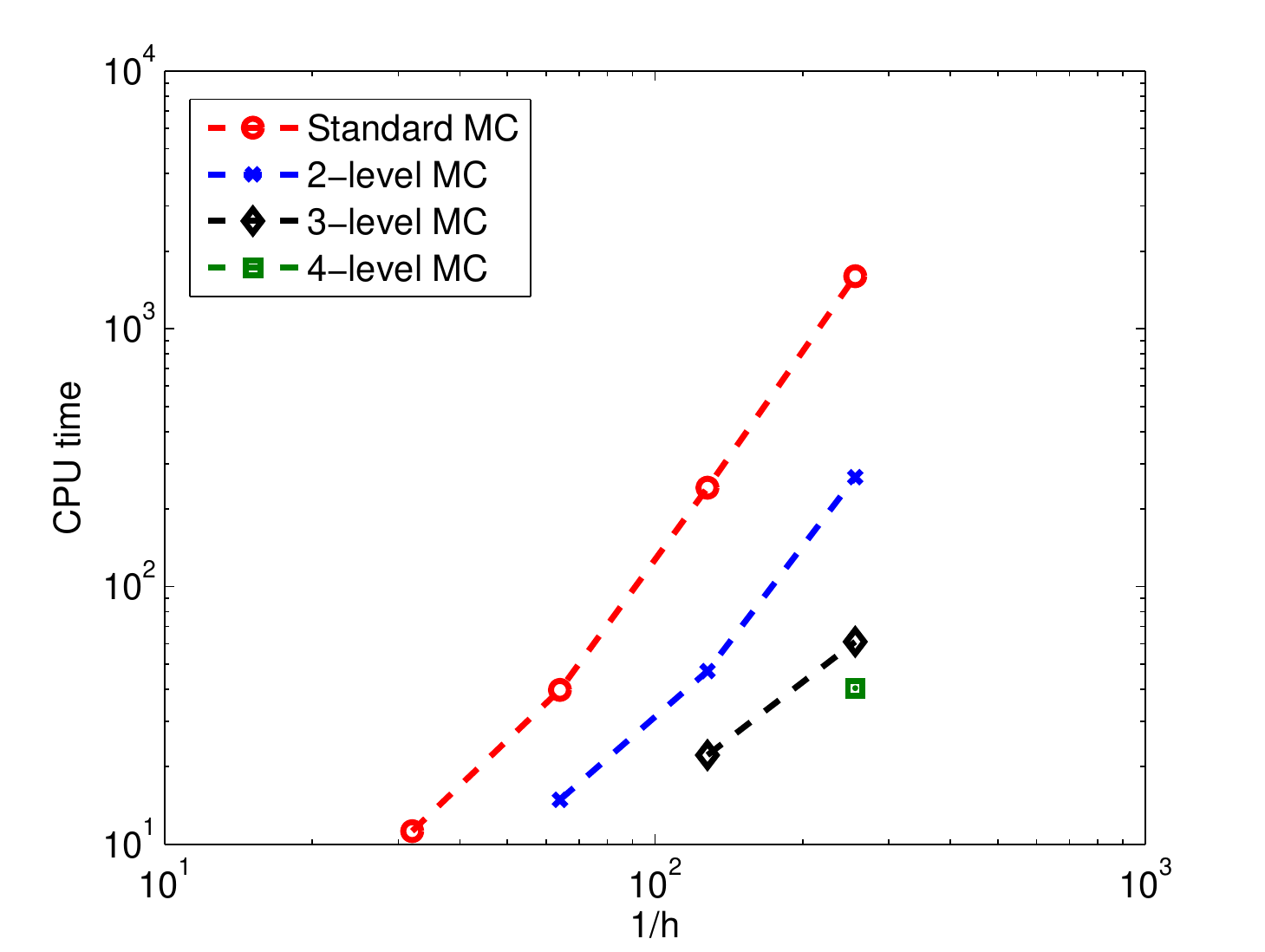}
\caption{Plot of CPU time versus $1/h$ for a fixed tolerance of $10^{-4}$ for the sampling
error. The quantity of interest is $k_{\textnormal{eff};h}$. The correlation function is as in
Figure~\ref{fig:exp_var_keff}.}
\label{fig:cpu:MLMC}
\end{figure}

\subsection{Travel time calculations}

Coming back to our motivating example in the introduction we now estimate statistics of
the time it takes particles to travel from a location in the computational
domain to its boundary.
To do this we use a very simple particle tracking model.
Having computed the Darcy velocity $\vq_\omega$ via \eqref{3.3} and neglecting molecular dispersion, the particle path
satisfies the initial value problem
\begin{equation}\label{tt:cont}
\frac{d\vec{x}_\omega}{dt}=\vq_\omega, \quad \vec{x}_\omega(0) = \vec{x}_0 \ ,
\end{equation}
where $\vec{x}_0 \in D$ denotes the starting point.
The traveltime $\tau_\omega \in [0,\infty)$ is the time when the particle hits the boundary, i.e. when
$\vec{x}_\omega(\tau_\omega) \in \partial D$ for the first time.

Due to the Picard-Lindel\"{o}f Theorem, a sufficient condition for problem \eqref{tt:cont} to
have a unique solution is that $\vq_\omega$ is Lipschitz continuous on $D$ as a function of $\vx$.
Assuming $u_\omega \in C^{1+t}(\bar{D})$ it follows that $\nabla
u_\omega \in C^{t}(\bar{D})^d$.
Combining this with Assumption A2, that is, $a_\omega \in C^{t}(\bar{D})$, for
some $0<t\leq 1$, we see that the Darcy velocity $\vq_\omega=-a_\omega 
\nabla u_\omega \in C^{t}(\bar{D})^d$ is H\"{o}lder continuous on $\bar{D}$ with
coefficient $t$. Thus, the Darcy velocity is Lipschitz only if $t=1$.
The required regularity result for the pressure $u_\omega \in
C^{1+t}(\bar{D})$ can be proved under Assumptions A1-A2 together with slightly stronger regularity conditions on the
source terms and on the boundary data, see \cite[Section 2.6]{aretha_thesis}.

To discretise problem \eqref{tt:cont} we replace $\vq_\omega$ by its FE approximation $\vq_{h,\omega}$ in \eqref{3.4}
and obtain
\begin{equation}\label{tt}
\frac{d\vec{x}_{\omega,h}}{dt}=\vq_{h,\omega}, \quad \vec{x}_{h,\omega}(0) = \vec{x}_0 \ .
\end{equation}
Again, the Picard-Lindel\"{o}f Theorem tells us that problem \eqref{tt} has a unique solution which
can be computed element by element over the triangulation $\mathcal{T}_h$ of $D$.
By following the particle through the domain $D$ and summing up the travel times in each
element we obtain the FE approximation $\tau_{h,\omega}$ to the actual travel time
$\tau_\omega$ for each realisation $a_\omega$ of the random permeability.

We use the 2D flow cell problem with exponential covariance \eqref{exp} and $\lambda=1$.
The particles are released at $\vec{x}_0=[0,0.5]^\top.$
The reference travel time $\tau_\star$ is computed on a grid with $n=256$ elements in each
spatial direction.
The results are depicted in Figure~\ref{fig:exp_var_tt} where we use again $N=5000$ samples
on each level.
Interestingly, we observe linear convergence in both plots suggesting that $\alpha = \beta$ 
in Theorem~\ref{MLMC-conv} in this case, in contrast to all the previous examples where 
$\beta = 2 \alpha$. This is due to a difference in weak convergence (in mean) needed 
in (M1) versus strong convergence (pathwise) needed in (M2) for the travel time. 
We have no rigorous theory to support these results as yet.

\begin{figure}[t]
\centering
\includegraphics[width=0.48\textwidth]{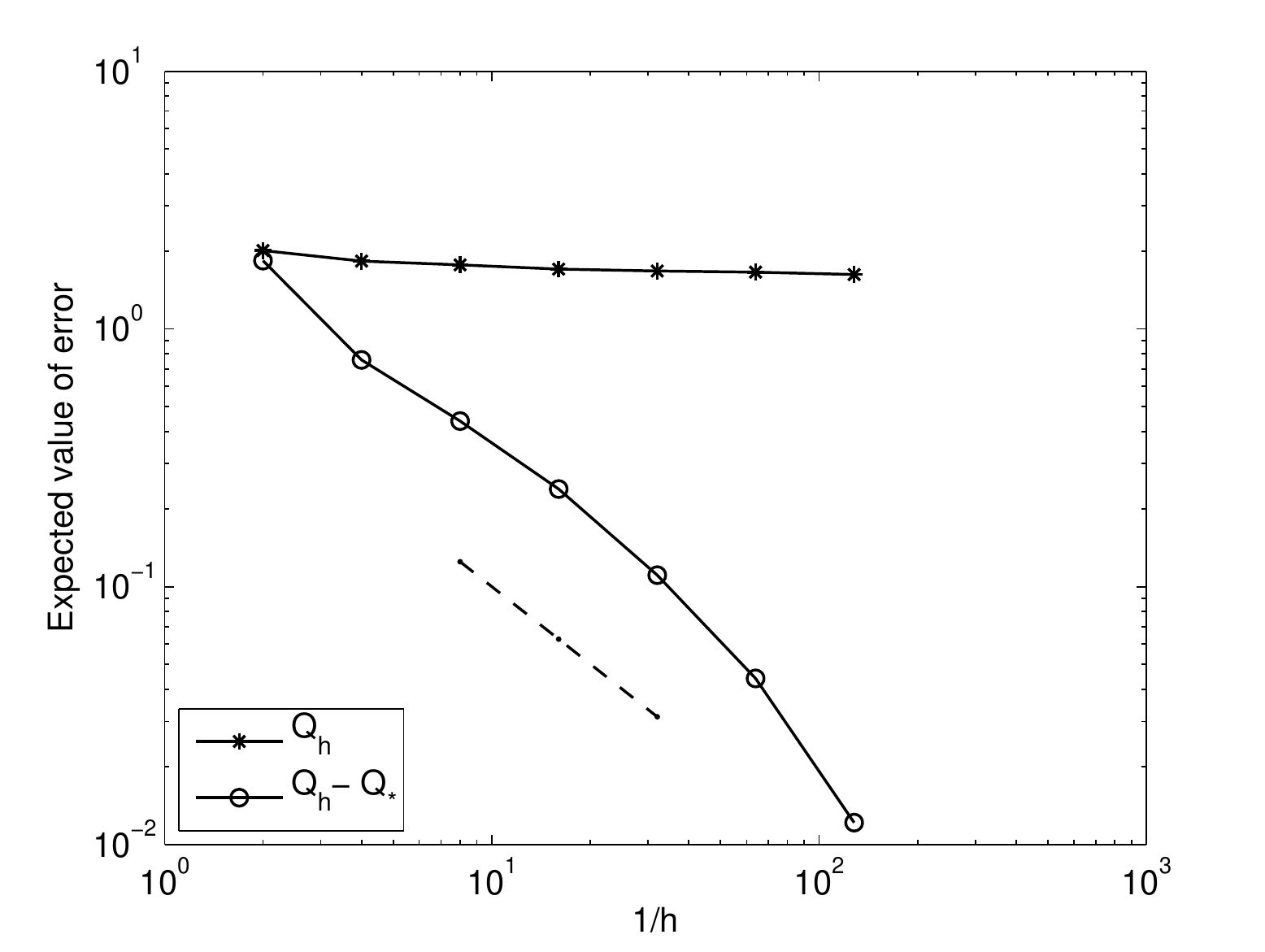}
\includegraphics[width=0.48\textwidth]{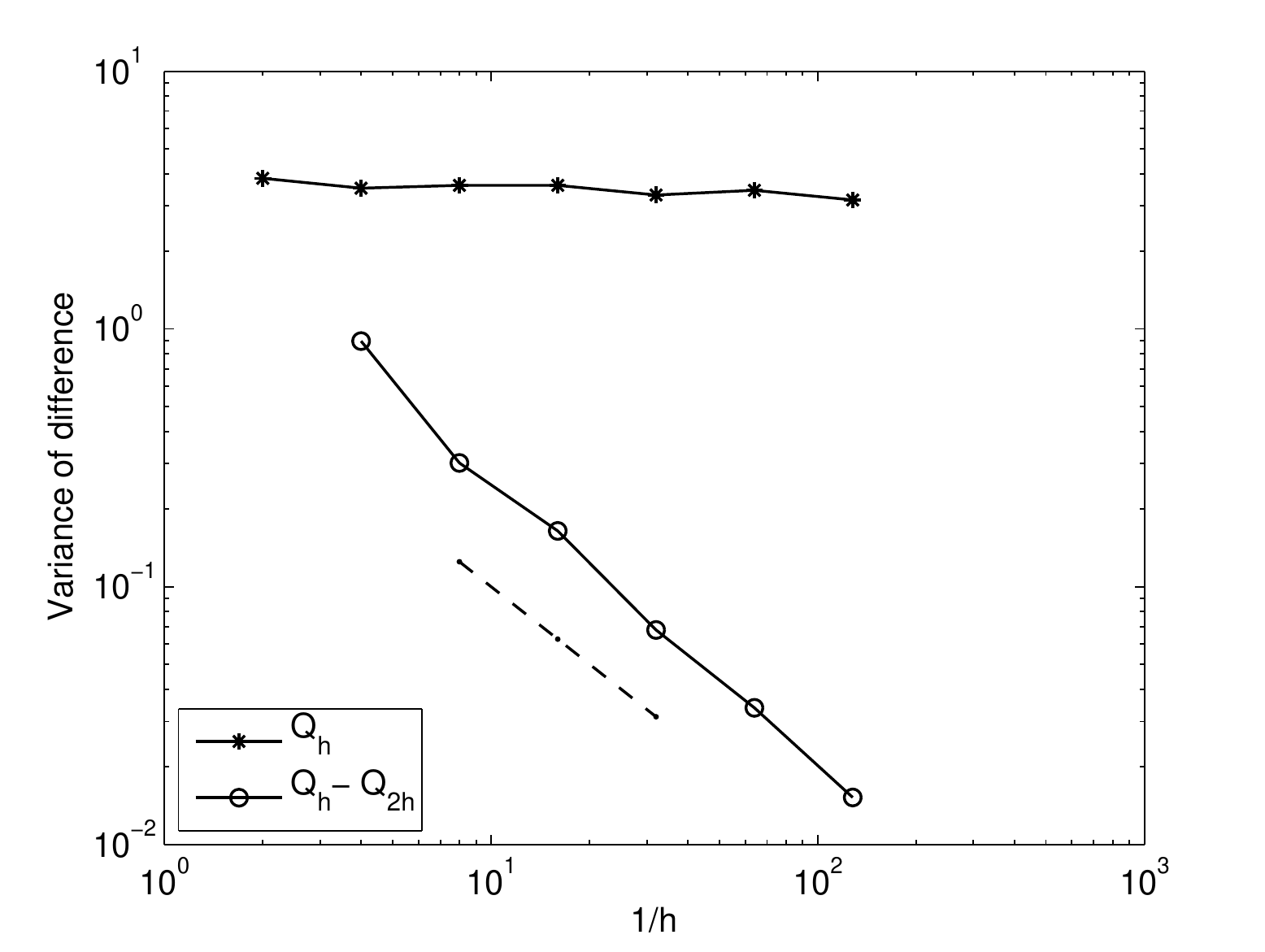}
\caption{Plot of
$|\mathbb{E}[\tau_h]|$ and
$|\mathbb{E}[\tau_h-\tau_{\star}]|$ (left), as well as $\mathbb{V}[\tau_h]$ and
$\mathbb{V}[\tau_h-\tau_{2h}]$ (right) with respect to the mesh size $h$ for $\rho=\rho_{\mathrm{\exp}}$ and
$\lambda=1$. The dashed line has slope $-1$.}
\label{fig:exp_var_tt}
\end{figure}

Finally, we compare the computational costs of the standard and multilevel MC estimator,
respectively, for the estimation of the expected value of the travel time $\tau$
where the sample error is fixed at $2.5 \times 10^{-3}$ (i.e. $\varepsilon = 5\sqrt{2} 
\times 10^{-2}$). The CPU timings (in seconds) shown in Figure~\ref{fig:cpu:MLMC:tt} 
are calculated using Matlab 7.13 on a quadcore Linux machine with a 3.30 GHz processor 
and 4 GByte of RAM with the optimal numbers of samples $N_\ell$ in \eqref{optimalNl}.
The efficiency of MLMC as compared to standard MC is clearly demonstrated.
On a mesh with $h=1/256$ the 3-level MLMC estimator takes only 5 minutes, but the single level MC
estimator takes 46 minutes.

\begin{figure}[t]
\centering
\includegraphics[width=0.48\textwidth]{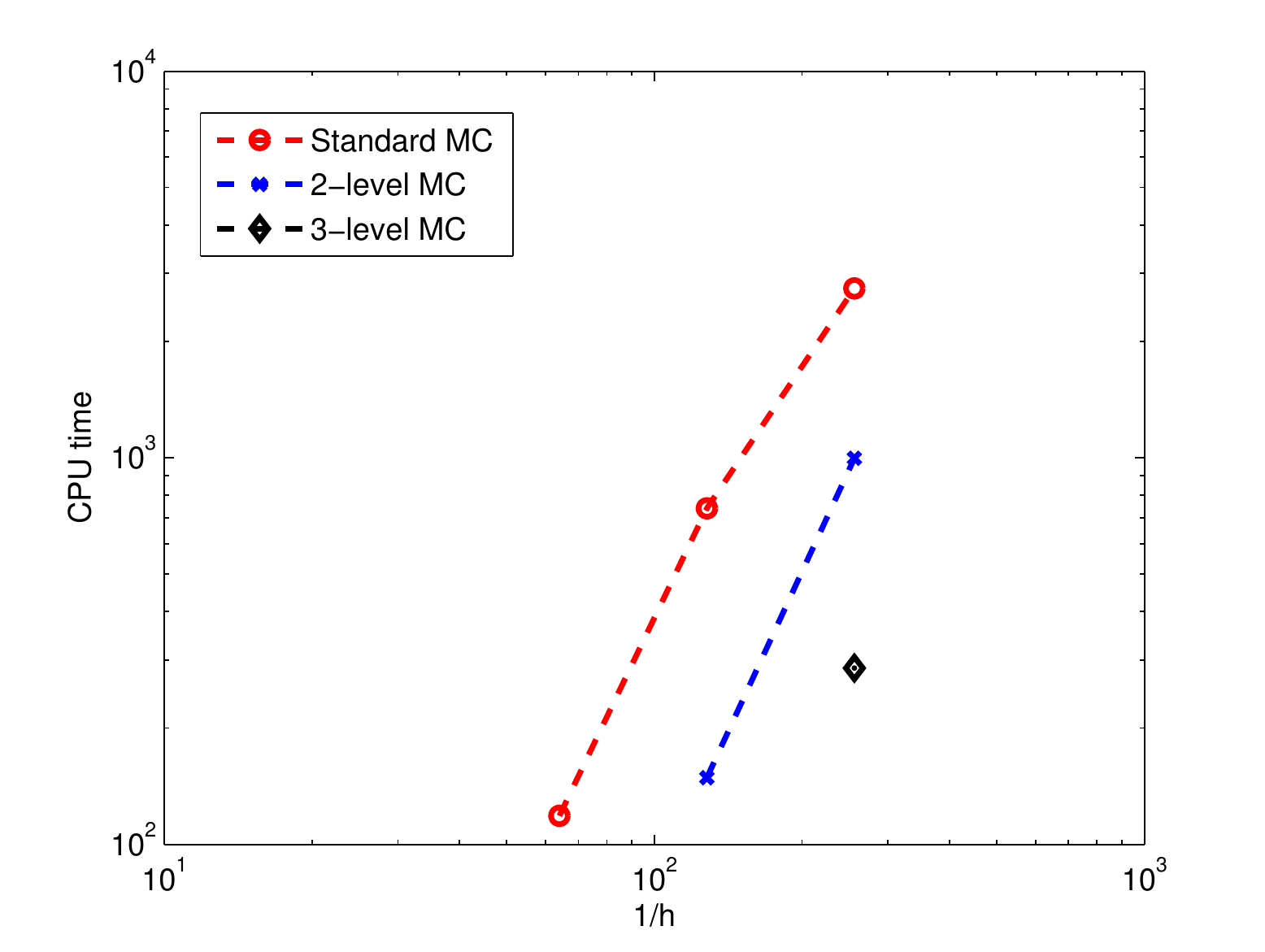}
\caption{Plot of CPU time versus $1/h$ for a
fixed tolerance $2.5\times 10^{-3}$ of the sampling error. The quantity of interest is
$\mathbb{E}[\tau_{h}]$. This plot is for $\rho=\rho_{\mathrm{\exp}}$ and $\lambda=1$.}
\label{fig:cpu:MLMC:tt}
\end{figure}

\section{Conclusions}
We studied a single phase flow problem in a random porous medium described by correlated lognormal distributions.
Realisations of the permeability are not uniformly bounded away from zero and infinity and they are in
general only H\"{o}lder-continuous with exponent $0<t\leq 1$.
We presented a mixed formulation of this problem and established the regularity of the Darcy velocity and the pressure. 
Using lowest order Raviart-Thomas mixed elements we proved finite element error bounds for the Darcy velocity
and the pressure, as well as a recovered pressure approximation. 
We showed that piecewise linear pressure recovery is only of practical interest for conductivity models with $t>1/2$.
Moreover, we proved error bounds for a class of linear functionals of the velocity.
This enabled us to bound the finite element approximation error for the effective permeability of a 2D flow cell.

Using the finite element error bounds we also proved convergence of a Multilevel Monte Carlo algorithm that is used to estimate the statistics of output quantities of interest associated with the flow problem, for example, the $H(\rdiv)$-norm of the velocity or the effective permeability.
In addition, we estimated the expected value of particle travel times to the boundary of the computational domain. As for all other output quantities in this work we observe that the MLMC estimator outperforms standard Monte Carlo. However, the gains are less substantial and there is still potential for improvement using ideas developed in the context of stochastic ordinary differential equations, e.g.~in \cite{Giles:2008b}, which merits further investigation.

\section*{Acknowledgements} 
We thank Daniele Boffi and Rodolfo Rodriguez for helpful comments. 
This research was supported by the Engineering and Physical Sciences Research Council, UK, under grant EP/H051503/1.

\small
\bibliographystyle{plain}
\bibliography{mixedFEbib}

\end{document}